\documentclass[a4paper,11pt]{article} 
\title{Large Deviations for a Non-Centered Wishart Matrix} 

\author{
\ Adrien Hardy 
\footnote{Institut de Math\'ematiques de Toulouse, Universit\'e de Toulouse, 31062 Toulouse, France.} \hspace{2 dd}\footnote{Department of Mathematics, Katholieke Universiteit Leuven, Celestijnenlaan 200 B,
3001 Leuven, Belgium. Email addresses: adrien.hardy@wis.kuleuven.be, arno.kuijlaars@wis.kuleuven.be} 
\,
and  
\ Arno B.J.  Kuijlaars 
\footnotemark[\value{footnote}] }

\usepackage[english]{babel}
\usepackage{anysize}
\marginsize{3.5cm}{3.5cm}{2.5cm}{2.5cm}
\usepackage{fancyhdr}
\usepackage{amsfonts}
\usepackage{amsmath} 
\usepackage{tikz,graphicx,subfigure,overpic}
\usepackage{color} 
\usepackage{amssymb} 
\usepackage{amsthm}
\usepackage[colorlinks=true, linkcolor=blue, citecolor=blue]{hyperref}
\usepackage[fixlanguage]{babelbib}

\numberwithin{equation}{section}

\newtheorem{theorem}{Theorem}[section]
\newtheorem{lemma}[theorem]{Lemma}
\newtheorem{corollary}[theorem]{Corollary}
\newtheorem{proposition}[theorem]{Proposition}
\newtheorem{Remark}[theorem]{Remark}
\newenvironment{remark}{\begin{Remark}\rm}{\end{Remark}}

\newtheorem{Example}[theorem]{Example}

\newtheorem{assumption}[theorem]{Assumption}

\newcommand{\tr}{\operatorname{Tr}}
\newcommand{\eq}{\begin{equation}}
\newcommand{\qe}{\end{equation}}
\newcommand{\Ma}{\mathcal{M}_1(\mathbb{R}_+)}
\newcommand{\Mc}{\mathcal{M}_{1/2}^{\sigma}\left(\mathbb{R}_-\right)}
\newcommand{\Mb}{\mathcal{M}_{1/2}(\mathbb{R}_-)}
\newcommand{\Me}{\mathcal{E}(\R_-)}
\newcommand{\Mf}{\mathcal{M}_{1/2}^{\sigma_N}\left(\mathbb{R}_-\right)}
\newcommand{\R}{\mathbb{R}}
\newcommand{\C}{\mathbb{C}}
\newcommand{\J}{\mathcal{J}}
\newcommand{\E}{\mathcal{E}}
\newcommand{\s}{\mathbb{S}}
\newcommand{\M}{\mathcal{M}}
\newcommand{\F}{\mathcal{F}}
\newcommand{\V}{\mathcal{V}}
\newcommand{\A}{\mathbb{A}}

\newcommand{\B}{\mathcal{B}}
\newcommand{\U}{\mathcal{U}}
\newcommand{\N}{\mathbb{N}}
\newcommand{\p}{\mathbb{P}}
\newcommand{\Supp}{{\rm Supp}}

\newcommand{\bs}{\boldsymbol}

\newcommand{\x}{{\bs{ x}}}

\begin{document}
\maketitle 

\begin{abstract}
We investigate an additive perturbation of a complex Wishart random matrix  and prove that a large deviation principle holds for the spectral measures. The rate function is associated to a vector equilibrium problem coming from logarithmic potential theory, which in our case is a quadratic map involving the logarithmic energies, or Voiculescu's entropies, of two measures in the presence of an external field and an upper constraint. The proof is based on a two type particles Coulomb gas representation for the eigenvalue distribution, which gives a new insight on why such variational problems should describe the limiting spectral distribution. This representation is available because of a Nikishin structure satisfied by the weights of the multiple orthogonal polynomials hidden in the background.   

\end{abstract}

\section{Introduction and statement of the results}
\subsection{Introduction}
The study of the large deviations for the spectral measures of large random matrices has started with the work \cite{BAG} of Ben Arous and Guionnet, and continued with many extensions, see e.g. \cite{BAZ, HP, ES, Bl, H}, which  now cover all the so-called unitary invariant matrix models, and actually  the larger class of $\beta$-ensembles.  The proof of such large deviation principles (LDPs) is based on the fact that an explicit and tractable expression is available for the joint eigenvalue distributions, which is  a consequence of the unitary invariance.  A common feature shared by these random matrix ensembles is that the rate functions governing such LDPs, which are maps on the space of probability measures, are given by the \emph{logarithmic energy} functional 
\eq
\label{logenergy}
\iint \log\frac{1}{|x-y|}d\mu(x)d\mu(y),
\qe
plus a linear term in the probability measure. The latter functional \eqref{logenergy} is the main object of study in 
logarithmic potential theory, and has moreover been interpreted up to a sign by Voiculescu as the \emph{free entropy}, a free probability equivalent of the Shannon's entropy in classical probability \cite{Vo}, see also  \cite{BS, HP, HMP}.

   More recently, much attention has been given to perturbed matrix models where one has broken the unitary invariance by the addition, or multiplication, of an external deterministic matrix, and also multi-matrix models.  It is a highly non-trivial problem to establish in full generality that a LDP still holds for such  matrix models, because of the complex dependence between the eigenvalues and eigenvectors. By developing an appropriate non-commutative It\^o calculus,  Cabanal-Duvillard and Guionnet  obtained a LDP upper bound for the spectral measures of a large class of matrix valued stochastic processes  \cite{CDG}. It has been later extended to a full LDP by Guionnet and  Zeitouni \cite{GZ1, GZ2}, and a LDP for perturbed or multi-matrix models actually follows by contraction principle. The price to pay for such a level of generality is a quite complicated rate function, but it is worth mentioning that it is known to reduce to the logarithmic energy in the unitary invariant case (i.e null perturbation), see \cite[Section 5.1]{CDG2}. 

In this work, we shall follow a different path and explore the large deviations of a  perturbed matrix model through its connection to multiple orthogonal polynomials (MOPs). Indeed, while the unitary invariant matrix models are known to be related to orthogonal polynomials \cite{Ko}, it has been observed by Bleher and Kuijlaars that perturbed matrix models benefit from a connection with MOPs \cite{BK}, in the sense that the average characteristic polynomial of the random matrix is a MOP with respect to appropriate weights and multi-index. Such relation  also holds for multi-matrix models \cite{DK}, see also \cite{Ku2} for a survey. On the other hand, the limiting zero distribution of certain classes of  MOPs can be described in terms of the solution of a vector equilibrium problem \cite{Apt,NS} : given  $d\geq1$ and a $d\times d$ real symmetric positive definite matrix $C=[c_{ij}] $, minimize the functional given by
\[
\sum_{1\leq i,j \leq d }c_{ij}\iint\log\frac{1}{|x-y|}d\mu_i(x)d\mu_j(y)
\]
plus linear terms in $(\mu_1,\ldots,\mu_d)$, when the vector of measures $(\mu_1,\ldots,\mu_d)$ runs over $\M_{m_1}(\Delta_1)\times\cdots\times\M_{m_d}(\Delta_d)$, or in some subset thereof. Here $\M_{m}(\Delta)$ stands for the set of Borel measures on $\Delta\subset\C$ with total mass $m$. For a general treatment  concerning vector equilibrium problems, see \cite{BKMW, HK}.

A natural question is then to seek if  the functionals associated to vector equilibrium problems should be involved as  large deviations rate functions. It is the aim of this work to answer affirmatively for a particular example that we present now.

\subsection{Non-centered Wishart random matrix}

The model we investigate here is a non-centered Wishart random matrix, which is an additive perturbation of the usual  Wishart model. Namely, let $X=[X_{ij}]$ be a $M\times N$ complex matrix filled with i.i.d  (non-centered) complex Gaussian random entries $X_{ij}\sim\mathcal{N}_{\C}(A_{ij},1/\sqrt{N})$, where $A=[A_{ij}]$ is a given deterministic  $M\times N$ complex matrix. One can equivalently endow the space $\M_{M,N}(\C)$ of $M\times N$ complex matrices with the probability distribution 
\eq
\label{Density}
d\p_N(X)=\frac{1}{Z_{M,N}}e^{-N\,{\rm Tr}\big((X- A)^*(X- A)\big)}dX,
\qe
where $Z_{M,N}$ is a normalization constant and $dX$ stands for the Lebesgue measure on $\M_{M,N}(\C)\simeq\R^{2MN}$. Without loss of generality, $A$ can be chosen in its singular value decomposition form. Note that, if $\U_N(\C)$ stands for the unitary group of $\C^N$, $\p_N$ is not invariant under the transformations $X\mapsto UXV^*$ for given $U\in\U_M(\C)$, $V\in\U_N(\C)$, except if  $A=0$.  

We are interested in the convergence and  deviations of the spectral measure
\eq
\label{spectralmeasure}
\mu^N=\frac{1}{N}\sum_{i=1}^N\delta(x_i),
\qe
where the $x_i$'s are the eigenvalues of the non-centered Wishart matrix $X^*X$ (or equivalently the squared singular values of $X$) with $X$ drawn according to  $\p_N$. It is a random variable taking its values in  $\M_1(\R_+)$,  that we equip with its weak topology.

This matrix model has been extensively studied in the statistic and signal processing literature (see e.g. \cite{SC} and references therein), and  Dozier and Silverstein described the limiting eigenvalue distribution for a large class of perturbations $A$ by means of a fixed point equation for its Cauchy-Stieltjes transform \cite{BDS1, BDS2}. Alternatively, the limiting eigenvalue distribution can be characterized in terms of the rectangular free  convolution introduced by Benaych-Georges \cite{BG}. On the other hand, the non-centered Wishart matrix model does not belong to the class of random matrices for which Guionnet and Zeitouni were able to extend  the LDP upper bound of  Cabanal-Duvillard and Guionnet into a full LDP  for the spectral measures $(\mu^N)_N$, and to prove such a LDP is in fact still an open problem.

In this work, we shall restrict our investigation to a particular case and assume that $M=N+\alpha$ where $\alpha$ is a non-negative integer, and consider for $a>0$ the particular type of perturbation
\eq
\label{perturbation}
A=\begin{bmatrix}\begin{matrix} \sqrt{a} & &  \\ &  \ddots&  \\ & &  \sqrt{a}  \end{matrix}  \\ \mathbf 0_{\alpha} \end{bmatrix}\in\mathcal{M}_{M,N}(\C).
\qe
As announced in the introduction, our goal is to establish a LDP for $(\mu^N)_N$ for which the rate function  involves a functional associated to a vector equilibrium problem, which itself describes the asymptotic distribution of the zeros of MOPs.  Let us now explains our intuition.

As it is classical for many random matrix models, one can embed such a non-centered Wishart matrix in a matrix valued stochastic process. Its  squared singular values then induce (up to a time change of variable) a process  of $N$ non-intersecting squared Bessel paths conditioned to start at $a>0$  and end at the origin. Kuijlaars, Mart\'inez-Finkelshtein and Wielonsky studied the particle system of a fixed-time marginal and established  a determinantal point process structure related to MOPs \cite{KMFW}, a so-called MOP ensemble \cite{Ku1}. Moreover, the limiting zero distribution of the MOPs involved in this particle system has been characterized by a vector equilibrium problem in \cite{KR}.  Combining these results, it is  likely  (see also \cite[Appendix]{KMFW}) that the spectral measure $\mu^N$  converges almost surely as $N\rightarrow\infty$ to a limiting distribution $\mu^*$ which is the first component  of the unique minimizer $(\mu^*,\nu^*)$ of the functional 
\begin{align}
\label{ratefunctionfake}
& \iint \log\frac{1}{|x-y|}d\mu(x)d\mu(y) - \iint \log\frac{1}{|x-y|}d\mu(x)d\nu(y)\nonumber\\
 &  \qquad\qquad+\iint\log\frac{1}{|x-y|}d\nu(x)d\nu(y) + \int \Big(x-2\sqrt{ax}\,\Big)d\mu(x)
\end{align}
when the vector of measures $(\mu,\nu)$ runs over $\M_1(\R_+)\times\Mc$, where we introduced the set of constrained  measures
\eq
\label{constrspace}
\Mc = \left\{ \nu\in\M_{1/2}(\R_-) : \; d\nu(x)\ll dx,\; \frac{d\nu}{dx}(x)\leq \frac{\sqrt{a}}{\pi}|x|^{-1/2}\right\}.
\qe
Here we use the notation
\[
\R_-=(-\infty,0],\qquad \R_+=[0,+\infty).
\]
Note that the functional \eqref{ratefunctionfake} is actually not well-defined for all  $(\mu,\nu)\in\M_1(\R_+)\times\Mc$, since the logarithmic energy \eqref{logenergy} can take the values $+\infty$ and $-\infty$ as well. We actually describe later an appropriate way to extend \eqref{ratefunctionfake} to the whole set $\M_1(\R_+)\times\Mc$, which is possible because it lies in the class of weakly admissible vector equilibrium problems introduced by the authors in \cite{HK}.

\subsection{Statement of the result}

The aim of this work is to show that such a  functional \eqref{ratefunctionfake}, once properly extended, is involved as a rate function governing a LDP for the spectral measures $(\mu^N)_N$. More precisely, our main result is the following.
\begin{theorem} 
\label{LDP}
The sequence of measures $(\mu^N)_N$ satisfies a LDP on $\M_1(\R_+)$ in the scale $N^2$ with good rate function \[
\inf_{\nu\in\M_{1/2}^{\sigma}(\R_-)}\J(\,\cdot\,,\nu)-\min\J,
\]
 where $\J$ is a well-defined extension of \eqref{ratefunctionfake} introduced in Section \ref{ratefunctionsection}. Namely, 
\begin{itemize}
\item[{\rm (a)}] The level set 
\[
\Big\{ \mu\in\M_1(\R_+) : \; \inf_{\nu\in\M_{1/2}^{\sigma}(\R_-)}\J(\mu,\nu)\leq \gamma\Big\}
\]
is compact for any $\gamma\in\R$.
\item[{\rm (b)}] $\J$ admits a unique minimizer $(\mu^*,\nu^*)$ on $\M_1(\R_+)\times\M_{1/2}^\sigma(\R_-)$.
\item[{\rm (c)}] For any closed set $\F\subset \M_1(\R_+)$,
\[
\limsup_{N\rightarrow\infty}\frac{1}{N^2}\log \p_N\Big(\mu^N\in\F\Big)\leq - \inf_{(\mu,\nu)\in\F\times\Mc}\Big\{\J(\mu,\nu)-\J(\mu^*,\nu^*)\Big\}.
\]
\item[{\rm (d)}] For any open set $\mathcal O\subset \M_1(\R_+)$,
\[
\liminf_{N\rightarrow\infty}\frac{1}{N^2}\log \p_N\Big(\mu^N\in\mathcal O\Big)\geq - \inf_{(\mu,\nu)\in\mathcal O\times\Mc}\Big\{\J(\mu,\nu)-\J(\mu^*,\nu^*)\Big\}.
\]
\end{itemize}
\end{theorem}

As a direct consequence of Theorem \ref{LDP} (b), (c) and the Borel-Cantelli Lemma, we obtain the almost sure convergence of $\mu^N$ towards $\mu^*$ in the weak topology of $\M_1(\R_+)$. Namely, if  $\p$ denotes  the measure induced by the product probability space $\bigotimes_N\big(\M_{M,N}(\C),\p_N\big)$, we have 

\begin{corollary}
\label{corollary}
\[
\p\Big(\mu^N \mbox{converges as } N\rightarrow\infty \mbox{ to $\mu^*$} \mbox{ in the weak topology of } \M_1(\R_+)\Big)= 1.
\]
\end{corollary}

\subsection{Generalizations and variations}

We now describe a few other  particle systems for which one can use the same approach as presented in this work to obtain a similar LDP statement.

\subsubsection{More general potentials}
The following generalization of the density distribution  \eqref{Density}  has been introduced by Desrosiers and Forrester in \cite{DF}   
\eq
\label{Density'}
\frac{1}{Z_{M,N}}e^{-N\,{\rm Tr}\big(V(X^*X)-{\rm Re}(X^*A)\big)}dX,
\qe
where $V:\R_+\rightarrow\R$ is a continuous function which is extended to Hermitian matrices by  functional calculus.  Indeed, by choosing  $V(x)=x$ we recover  the non-centered Wishart matrix model. Now, if we take $A$ as in \eqref{perturbation} and we assume that $V$ satisfies the growth condition
\[
\liminf_{x\rightarrow+\infty} \frac{V(x)-2\sqrt{ax}}{2\log(x)}>1,
\]
then one can follow the methods developed in this work without substantial change to show an analogue of Theorem \ref{LDP} and Corollary \ref{corollary}, where $\J$ is replaced by a well-defined extension in the sense of \cite{HK} (see also Section \ref{ratefunctionsection}) over $\M_1(\R_+)\times\M^\sigma_{1/2}(\R_-)$ of  the functional 
\begin{align*}
& \iint \log\frac{1}{|x-y|}d\mu(x)d\mu(y) - \iint \log\frac{1}{|x-y|}d\mu(x)d\nu(y)\nonumber\\
 &  \qquad\qquad+\iint\log\frac{1}{|x-y|}d\nu(x)d\nu(y) + \int \Big(V(x)-2\sqrt{ax}\,\Big)d\mu(x).
\end{align*}

\subsubsection{Rescaling the parameter $\alpha$}
Observe that in our setting we have $M/N\rightarrow 1$ as $N\rightarrow\infty$. A natural question would be to investigate the case where one performs the rescaling $\alpha\mapsto\alpha N$, so that $M/N\rightarrow 1+\alpha$ as $N\rightarrow\infty$ with $\alpha\geq 0$. It turns out that the  approach we develop below is still well-suited for this case, but  requires  more involved asymptotic estimates for Bessel functions and its zeros than the ones we use  in this paper. These asymptotic estimates are actually provided by \cite[(9.7.7)]{AS} and \cite[(9.5.22)]{AS}, and they would lead to statements similar to Theorem \ref{LDP} and Corollary \ref{corollary}, where $\J$ is replaced by a well-defined extension  of
\begin{align*}
& \iint \log\frac{1}{|x-y|}d\mu(x)d\mu(y) - \iint \log\frac{1}{|x-y|}d\mu(x)d\nu(y) +\iint\log\frac{1}{|x-y|}d\nu(x)d\nu(y) \nonumber\\
 &  + \int \left(x-\alpha\log(x)-\sqrt{4ax+\alpha^2}+ \alpha\log\left(\alpha+\sqrt{4ax+\alpha^2}\right)\,\right)d\mu(x)
\end{align*}
where  the space $\Mc$ is now defined by
\[
\Mc = \left\{ \nu\in\M_{1/2}\Big((-\infty,-\tfrac{\alpha^2}{4a}]\Big) : \; d\nu(x)\ll dx,\; \frac{d\nu}{dx}(x)\leq \frac{\sqrt{4a|x|-\alpha^2}}{2\pi|x|}\right\}.
\]
This is also the functional obtained in \cite{KR} when describing the limiting zero distribution of the associated MOPs. Nevertheless, in this setting the proof becomes  more technical, and we chose to restrict ourselves to the non-rescaled model for the sake of clarity.

\subsubsection{Non-intersecting Bessel paths with one positive starting and ending point}
In \cite{DKRZ}, Delvaux, Kuijlaars, Rom\'an and Zhang investigated a system of $N$ non-intersecting squared Bessel paths conditioned to start from $a>0$ at time $t=0$ and to end at $b>0$ when $t=1$. It is actually not known if such model is related to a random matrix ensemble. We note that at fixed time $0<t<1$, it is easy to express the particle distribution as the marginal distribution of a Coulomb gas involving three different type particles by combining \cite[Section 2.5]{DKRZ} with the computations we present in Section 2.  As a consequence, if we introduce the functional $\J$ to be the well-defined extension of 
\begin{align}
& \iint \log\frac{1}{|x-y|}d\mu(x)d\mu(y) - \iint \log\frac{1}{|x-y|}d\mu(x)d\nu(y)\\
 & \qquad - \iint\log\frac{1}{|x-y|}d\mu(x)d\eta(y)+\iint\log\frac{1}{|x-y|}d\nu(x)d\nu(y) \nonumber\\
 & \qquad \qquad +\iint\log\frac{1}{|x-y|}d\eta(x)d\eta(y) + \int \left(\frac{x}{t(1-t)}-\frac{2\sqrt{ax}}{t}-\frac{2\sqrt{bx}}{1-t}\,\right)d\mu(x)\nonumber
\end{align}
in the sense of \cite{HK}  (see also Section \ref{ratefunctionsection}), where $\mu\in\M_1(\R_+)$, and $\nu,\eta\in\M_{1/2}(\R_-)$ satisfy 
\[
\frac{d\nu}{dx}(x)\leq\frac{\sqrt{a}}{\pi t}|x|^{-1/2},\qquad  \frac{d\eta}{dx}(x)\leq\frac{\sqrt{b}}{\pi (1-t)}|x|^{-1/2},
\]
then a  LDP similar to the one of the non-centered Wishart matrix holds where the rate function is given by $\J-\min\J$ after taking the infimum over all constrained measures $\nu$ and $\eta$. Indeed, there is no interaction between the particles associated to $\nu$ and $\eta$, and then both $\nu$ and $\eta$ interact with $\mu$ exactly in the same way that  $\nu$ interacts with  $\mu$ in the present work, so that a LDP can be established with no extra work from the ingredients of the proof we present below.  

\subsection{Open problems}
There are  other matrix models for which it is established that the limiting mean spectral distribution is characterized in terms of the solution of a vector equilibrium problem, thanks to their connection with MOPs and a Riemann-Hilbert asymptotic analysis. Examples can be found in the Hermitian matrix model with an external source \cite{BDK} and the two-matrix model \cite{DKM, DGK}. Nevertheless, it is not clear to the authors how to strengthen such convergence results to get LDPs.

Another question of interest would be to see if the rate function introduced by Cabanal-Duvillard and Guionnet in \cite{CDG} reduces for such matrix models to the functional of a vector equilibrium problem.

\subsection{Strategy of the proof}
In Section \ref{coulombgas}, we show that the joint eigenvalue distribution of the non-centered Wishart matrix  is the marginal distribution of a 2D Coulomb gas with two type particles. The first type of particles are living on $\R_+$ and are exactly the eigenvalues of our matrix model. The second type of particles are abstract ones and live on a $N$-dependent discrete subset of $\R_-$. They moreover attract the first type of particles, expressing the effect of the perturbation. This provides an  insight as to why a functional like \eqref{ratefunctionfake} should be involved as a rate function. To prove such statement, we first describe in Section \ref{MOPensemble} the eigenvalue distribution as a MOP ensemble, and then make use of the Nikishin structure satisfied by the weights associated to the polynomials in Section \ref{nikishin}.

In Section \ref{general}, we investigate the generalized particle system of the whole Coulomb gas for which we state a LDP, see Theorem \ref{generalLDP}. Theorem \ref{LDP} then follows by contraction principle, as described by Corollary \ref{contractionprinciple}. We  define the rate function in Section \ref{ratefunctionsection}, which is a  proper extension of \eqref{ratefunctionfake}, by following the approach of \cite{HK}. From the discrete character of the particles on $\R_-$, a discussion provided in Section \ref{discretness} explains why the constraint set $\Mc$  naturally appears in the variational problem.    

In Section \ref{proof}, we provide a proof of Theorem \ref{generalLDP}. The two main difficulties are the absence of  confining potential acting on the particles living on $\R_-$, and the possible contact of the two different type of particles at the origin. Concerning the lack of confining potential, we  follow an approach developed in \cite{H} by one of the authors and perform a well-adapted compactification procedure. For the contact at the origin, we  isolate the induced singularity and use the discrete character of the particles on $\R_-$ to control it, see the proof of Proposition \ref{LDPUB} and particularly  Lemma \ref{unifint}.

\begin{remark}
From now, we will assume that $N$ is even to simplify the notations and the presentation, but our proof easily adapts to the general case by replacing $N/2$ by $\lceil N/2\rceil$ or $\lfloor N/2 \rfloor$, and also $1/2$ by $\lceil N/2\rceil/N$ or $\lfloor N/2 \rfloor/N$, where it is necessary.
\end{remark}

\section{A 2D Coulomb gas of two type particles}
\label{coulombgas}

In this section we show that the joint eigenvalue distribution is the marginal distribution of a Coulomb  gas having two types of particles. Such a representation follows from a  particular type of MOP ensemble structure satisfied by the eigenvalues that  we describe now. 

\subsection{Multiple orthogonal polynomial ensemble}
\label{MOPensemble}
We first show that the eigenvalues form a MOP ensemble in the sense of \cite{Ku2}, a particular type of Borodin's biorthogonal ensemble \cite{Bo}. For that, introduce the Vandermonde determinant
\eq
\label{vdm}
\Delta_N(\x) = \det \big[x_j^{i-1}\big]_{i,j=1}^N = \prod_{1\leq i < j \leq N}\big(x_j-x_i\big).
\qe
For $\alpha\geq0$ and $a>0$, consider moreover the weight function
\eq
\label{weight}
w_{\alpha,N}(x)=x^{\alpha/2}I_{\alpha}\big(2N\sqrt{ax}\,\big)e^{-Nx},\qquad x\in\R_+,
\qe
 where  we introduced the modified Bessel function of the first kind 
\eq
\label{BF}
I_\alpha(x)=\sum_{k=0}^\infty\frac{1}{k!\Gamma(k+\alpha+1)}\left(\frac{x}{2}\right)^{2k+\alpha}.
\qe
We mention that these weights have been introduced and  studied by Coussement and Van Assche \cite{CVA, CVA2}. We now prove the following.

\begin{lemma}
\label{MOPE}
The  joint  probability density for the eigenvalues $\x=(x_1,\ldots,x_N)\in\R_+^N$ of $X^*X$, when $X$ is drawn according to \eqref{Density}, is a $(N/2,N/2)$-MOP ensemble with weights $w_{\alpha,N}$ and $w_{\alpha+1,N}$, that is given by
\eq
\label{densityev}
\frac{1}{Z_N}\Delta_N(\bs x)\det \left[\begin{matrix} \Big\{ x_j^{i-1}w_{\alpha,N}(x_j)\Big\}_{i,j=1}^{N/2,N}\\ \\ \Big\{ x_j^{i-1}w_{\alpha+1,N}(x_j)\Big\}_{i,j=1}^{N/2,N} \end{matrix}\right]
\qe
where $Z_N$ is a new normalization constant.
\end{lemma}

\begin{proof}
We perform a singular value decomposition of $X$, that is we write $X=U_1X_{\rm{diag}}\,U_2$ for unitaries $U_1\in\U_M(\C)$ and $U_2\in\U_N(\C)$  
with
\[
X_{\rm{diag}}=\begin{bmatrix}\sqrt{x_1} & &  \\ &  \ddots&  \\ & &  \sqrt{x_N}  \\ &  \boldsymbol 0_{\alpha}&  \end{bmatrix},
\]
and note that 
\begin{align}
\label{devtrace}
\tr\big((X-A)^*(X-A)\big)= & \;\tr (X^*X)-\tr(XA^*+AX^*)+Na    \\
= &\; \sum_{i=1}^Nx_i- \tr\big(X_{{\rm diag}}U_2A^*U_1^*+(X_{{\rm diag}}U_2A^*U_1^*)^*\big)+Na\nonumber.
\end{align}
By integrating over the unitary groups, it follows from the Weyl integration formula \cite[Section 4.1]{AGZ} and  \eqref{devtrace}, that the probability density  for the $x_i$'s induced by $\p_N$  is given by
\eq
\label{densitya}
\frac{1}{Z_N}\Delta^2_N(\bs x)\prod_{i=1}^N x_i^{\alpha}e^{-Nx_i}\int_{\U_{M}(\C)}\int_{\U_N(\C)}e^{N\tr\big(X_{\rm{diag}}UA^*V^*+(X_{\rm{diag}}UA^*V^*)^*\big)}dU\,dV,
\qe
where $dU$ (resp. $dV$) stands for the Haar measure of $\U_{N}(\C)$ (resp. $\U_M(\C)$) and $Z_N$ is a new normalization constant. Note that one can assume the $x_i$'s to be distinct since this holds almost surely. Consider
\[
B=\begin{bmatrix}\sqrt{b_1} & &  \\ &  \ddots&  \\ & &  \sqrt{b_N} \\  & \mathbf 0_{\alpha}&  \end{bmatrix}\in\mathcal{M}_{M,N}(\C)
\]
with $a \leq b_1 < \cdots < b_N\leq a+1$. Then we have the following formula for the matrix integral  \cite[Section 3.2]{ZZ}
\begin{multline}
\label{ZZ}
\int_{\U_{M}(\C)}\int_{\U_{N}(\C)}e^{N {\rm Tr}\big(X_{\rm{diag}}UB^*V^*+(X_{\rm{diag}}UB^*V^*)^*\big)}dU\,dV = \\
 c_N\left(\,\prod_{i=1}^N\frac{1}{(b_ix_i)^{\alpha/2}}\right)\frac{\det\Big[I_{\alpha}\big(2N\sqrt{b_ix_j}\,\big)\Big]_{i,j=1}^N}{\Delta_N(\bs x)\Delta_N(\bs b)},
\end{multline}
where $c_N$ is a positive number which does not depend on $\bs x$ nor $\bs b$. By continuity of the left-hand side of \eqref{ZZ} in the   $b_i$'s, we then obtain that \eqref{densitya} is proportional to
\[
\label{densityb}
\lim_{b_{N}\rightarrow a}\cdots \lim_{b_1\rightarrow a}\;\left\{\frac{\Delta_N(\bs x)}{\Delta_N(\bs b)}\det\Big[x_j^{\alpha/2}I_{\alpha}\big(2N\sqrt{b_ix_j}\,\big)e^{-Nx_j}\Big]_{i,j=1}^N\right\},
\]
and thus to
\eq
\label{densityc}
\Delta_N(\bs x)\det\left[\frac{\partial^{i-1}}{\partial b^{i-1}}\left\{x_j^{\alpha/2}I_{\alpha}\big(2N\sqrt{bx_j}\,\big)e^{-Nx_j}\right\}\Big|_{b=a}\right]_{i,j=1}^N
\qe
by l'H\^opital Theorem. Finally, using for $x>0$ the relations \cite[P79]{W}
\[
\frac{d}{dx}I_{\alpha}(x)=I_{\alpha+1}(x)+\frac{\alpha}{x} I_\alpha(x),\qquad \frac{d}{dx}I_{\alpha+1}(x) = I_{\alpha}(x)-\frac{\alpha+1}{x} I_{\alpha+1}(x),
\] 
it is easily shown inductively that the linear space spanned by the functions
\[
x\mapsto \frac{\partial^{i-1}}{\partial b^{i-1}}\left\{x^{\alpha/2}I_{\alpha}\big(2N\sqrt{bx}\,\big)e^{-Nx}\right\}\Big|_{b=a}, \qquad i=1,\ldots,N,
\]
matches with the one spanned by
\[
x\mapsto x^{i-1} w_{\alpha,N}(x) , \quad x\mapsto x^{i-1} w_{\alpha+1,N}(x), \qquad i=1,\ldots,N/2.
\]
This ends the proof of Lemma \ref{MOPE}.
\end{proof}

\begin{remark} Although the parameter $\alpha$ associated to the  matrix model is a non-negative integer, the distribution \eqref{densityev} still makes sense for non-negative real $\alpha$. In fact, in  the proofs we provide later, it will not matter whether $\alpha$ is an integer or not. Thus, if one considers the measures $\mu^N$ \eqref{spectralmeasure} associated to $x_i$'s drawn according  to \eqref{densityev} with real $\alpha\geq 0$, the LDP from Theorem \ref{LDP} continues to hold.
\end{remark}

\subsection{Nikishin system}
\label{nikishin}

We now describe a property satisfied by the weights $w_{\alpha,N}$ and $w_{\alpha+1,N}$, a so-called Nikishin structure, and obtain as a consequence an exact Coulomb gas representation for the eigenvalues, see Proposition \ref{loggas}. The reader curious about Nikishin  systems   should have a look at \cite{NS} (where they are called MT systems).

More precisely, it turns out that the ratio of the weights is (almost) the Cauchy transform of some measure, a fact which has already been observed  \cite[Theorem 1]{CVA}. We now make this result slightly more precise with an alternative simple proof. Consider the sequence 
\[
0<j_{\alpha,0}<j_{\alpha,1}<j_{\alpha,2}<\cdots
\] 
of the  positive zeros of the Bessel function of the first kind $J_\alpha$,  a rotated version of $I_\alpha$, i.e 
\eq
\label{relationbesselfunctions}
J_{\alpha}(x)=e^{i\pi\alpha/2}I_{\alpha}(-ix), \qquad x\in\R_+,
\qe
and introduce for each $N$  the sequence of negative  numbers
\eq 
\label{akN}
a_{k,N}=-\left(\frac{j_{\alpha,k}}{2\sqrt{a}N}\right)^2, \qquad k\geq 0.
\qe
We then set for convenience 
\eq
\label{AN}
\A_N=\Big\{ a_{k,N}:\; k\geq 0\Big\}
\qe
and consider the associated normalized counting measure
\eq
\label{sigmaN}
\sigma_N=\frac{1}{N}\sum_{u\in\A_N}\delta(u).
\qe 
The weights $w_{\alpha,N}$ and $w_{\alpha+1,N}$ then satisfy the following relation.
\begin{lemma} 
\label{Nikishin} For all $\alpha\geq 0$ and $a>0$,
\[
\frac{w_{\alpha+1,N}}{w_{\alpha,N}}(x)= \frac{x}{\sqrt{a}}\int \frac{\;d\sigma_N(u)}{x-u}, \qquad x\in\R_+.
\]
\end{lemma}
\begin{proof}
Up to a change of variable, this relation is nothing else than the Mittag-Leffler expansion
\[
\frac{I_{\alpha+1}}{I_{\alpha}}(x)=2x\sum_{k=0}^\infty\frac{1}{x^2+j_{\alpha,k}^2}
\]
which is itself provided by \cite[P61]{EMOT} together with the relation \eqref{relationbesselfunctions}.
\end{proof}

Lemma \ref{Nikishin} is in fact the key to express the eigenvalue density \eqref{densityev} as the marginal distribution of a two type particles Coulomb gas.  Namely, if we introduce a Vandermonde-like product for $(\bs x, \bs u)\in\R_+^{N}\times\R_-^{N/2}$
\eq
\label{VdMmixte}
\Delta_{N,N/2}(\bs x, \bs u) = \prod_{i=1}^{N}\prod_{j=1}^{N/2}\big(x_i-u_j\big)= \prod_{i=1}^{N}\prod_{j=1}^{N/2}\big|x_i-u_j\big|,
\qe
then the following Proposition holds.

\begin{proposition}
\label{loggas}
The probability density   \eqref{densityev} admits the following representation
\[
\frac{1}{Z_N}\int_{\R_-^{N/2}}\frac{\Delta_N^2(\bs x)\Delta_{N/2}^2(\bs u)}{\Delta_{N,N/2}(\bs x,\bs u)}\prod_{i=1}^Nw_{\alpha,N}(x_i)\prod_{i=1}^{N/2}|u_i|d\sigma_N(u_i)
\]
where $Z_N$ is a new normalization constant.
\end{proposition}

The proof we present now is inspired from the proof of \cite[Theorem 2]{CVA}.

\begin{proof} 
Recall that the  density \eqref{densityev} is proportional to
\eq
\label{F1}
\Delta_N(\bs x)\det \left[\begin{matrix} \Big\{ x_j^{i-1}w_{\alpha,N}(x_j)\Big\}_{i,j=1}^{N/2,N}\\ \\ \Big\{ x_j^{i-1}w_{\alpha+1,N}(x_j)\Big\}_{i,j=1}^{N/2,N} \end{matrix}\right].
\qe
We first perform the factorization  
\eq
\label{F2}
\det \left[\begin{matrix} \Big\{ x_j^{i-1}w_{\alpha,N}(x_j)\Big\}_{i,j=1}^{N/2,N}\\ \\ \Big\{ x_j^{i-1}w_{\alpha+1,N}(x_j)\Big\}_{i,j=1}^{N/2,N} \end{matrix}\right] =
\det \left[\begin{matrix} \Big\{ x_j^{i-1}\Big\}_{i,j=1}^{N/2,N}\\ \\ \Big\{ x_j^{i-1}\,\displaystyle \frac{ w_{\alpha+1,N}}{w_{\alpha,N}}(x_j)\Big\}_{i,j=1}^{N/2,N} \end{matrix}\right]\prod_{i=1}^Nw_{\alpha,N}(x_i)
\qe
and then use Lemma \ref{Nikishin} to obtain 
\begin{multline}
\det \left[\begin{matrix} \Big\{ x_j^{i-1}\Big\}_{i,j=1}^{N/2,N}\\ \\ \Big\{ x_j^{i-1}\,\displaystyle \frac{ w_{\alpha+1,N}}{w_{\alpha,N}}(x_j)\Big\}_{i,j=1}^{N/2,N} \end{matrix}\right] \\
= 
 (\sqrt{a}\,)^{-N/2}\int_{\R_-^{N/2}}\det \left[\begin{matrix} \Big\{ x_j^{i-1}\Big\}_{i,j=1}^{N/2,N} \\\Bigg\{\,\displaystyle\frac{x_j^{i}}{x_j-u_i}\,\Bigg\}_{i,j=1}^{N/2,N} \end{matrix}\right] \prod_{i=1}^{N/2}d\sigma_N(u_i).
\end{multline}
Provided with the identity
\[
\frac{x^i}{x-u}=\frac{u^i}{x-u}+\sum_{k=0}^{i-1}x^ku^{i-k+1},
\]
the multilinearity of the determinant gives
\begin{align}
\label{F3}
\det \left[\begin{matrix} \Big\{ x_j^{i-1}\Big\}_{i,j=1}^{N/2,N} \\\Bigg\{\,\displaystyle\frac{x_j^{i}}{x_j-u_i}\,\Bigg\}_{i,j=1}^{N/2,N} \end{matrix}\right] 
= &
\det \left[\begin{matrix} \Big\{ x_j^{i-1}\Big\}_{i,j=1}^{N/2,N} \\ \Bigg\{\,\displaystyle\frac{u_i^{i}}{x_j-u_i}\,\Bigg\}_{i,j=1}^{N/2,N} \end{matrix}\right]  \nonumber  \\
= & 
\det \left[\begin{matrix} \Big\{ x_j^{i-1}\Big\}_{i,j=1}^{N/2,N} \\ \Bigg\{\,\displaystyle\frac{1}{x_j-u_i}\,\Bigg\}_{i,j=1}^{N/2,N} \end{matrix}\right] \prod_{i=1}^{N/2}u_i^{i}.
\end{align}
Now, the well-known identity for mixed Cauchy-Vandermonde determinant, see e.g.\ \cite[Lemma 3]{CVA}, yields
\eq
\label{F4}
\det \left[\begin{matrix} \Big\{ x_j^{i-1}\Big\}_{i,j=1}^{N/2,N} \\ \Bigg\{\,\displaystyle\frac{1}{x_j-u_i}\,\Bigg\}_{i,j=1}^{N/2,N} \end{matrix}\right]= \pm \frac{\Delta_N(\bs x)\Delta_{N/2}(\bs u)}{\Delta_{N,N/2}(\bs x, \bs u)},
\qe
where the sign only depends on $N$. Combining \eqref{F1}--\eqref{F4}, we obtain that \eqref{densityev} is proportional to
\eq
\label{F5}
\int_{\R_-^{N/2}}\frac{\Delta^2_N(\bs x)\Delta_{N/2}(\bs u)}{\Delta_{N,N/2}(\bs x, \bs u)}\prod_{i=1}^Nw_{\alpha,N}(x_i)\prod_{i=1}^{N/2}u_i^{i}\,d\sigma_N(u_i).
\qe
By summing the integrand of  \eqref{F5} over all possible permutations of the $u_i's$ and using the definition \eqref{vdm} of the Vandermonde determinant, we obtain that \eqref{F5} is proportional to 
\[
\int_{\R_-^{N/2}}\frac{\Delta^2_N(\bs x)\Delta^2_{N/2}(\bs u)}{\Delta_{N,N/2}(\bs x, \bs u)}\prod_{i=1}^Nw_{\alpha,N}(x_i)\prod_{i=1}^{N/2}|u_i|\,d\sigma_N(u_i).
\]
Since for every $(\bs x , \bs u)\in\R_+^N\times\A_N^{N/2}$ the quantity
\[
\frac{\Delta^2_{N/2}(\bs x)\Delta_{N/2}^2(\bs u)}{\Delta_{N,N/2}(\bs x,\bs u)}\prod_{i=1}^Nw_{\alpha,N}(x_i)\prod_{i=1}^{N/2}|u_i|
\]
is non-negative (and not identically zero), the new normalization constant $Z_N$ has to be positive. The proof of Proposition \ref{loggas} is therefore complete.
\end{proof}
In the next section, we perform a large deviations investigation for the whole Coulomb gas system.

\section{A LDP for the generalized particle system}
\label{general}
On the basis of the preceding analysis, we investigate in this section the probability distribution on $\R_+^N\times \R_-^{N/2}$
\eq 
\label{generaldistr}
\frac{1}{Z_{N}}\frac{\Delta_{N}^{2}\left(\bs x\right)\Delta_{N/2}^{2}\left(\bs u\right)}{\Delta_{N,N/{2}}\left(\bs x,\bs u\right)}\prod_{i=1}^{N}e^{-NV_N(x_i)}dx_i\prod_{i=1}^{N/2}|u_i|d\sigma_N(u_i)
\qe 
where, with $w_{\alpha,N}$ defined in \eqref{weight}, we introduced for convenience 
\eq
\label{VN}
V_N(x)=-\frac{1}{N}\log w_{\alpha,N}(x), \qquad x\in\R_+.
\qe
The measure $\sigma_N$ has been defined in \eqref{sigmaN}, and $Z_N$ is a normalization constant. 

Consider the  empirical measure for the second type particles
\eq
\nu^{N}=\frac{1}{N}\sum_{i=1}^{N/{2}}\delta(u_{i})
\qe
where the $u_i$'s are distributed according to \eqref{generaldistr} and note that the random vector of measures $(\mu^N,\nu^N)$ takes values in $\Ma\times\Mb$, that we equip with the  product topology. Our aim is to establish a LDP for $\big((\mu^{N},\nu^{N})\big)_N$, from which follows a LDP for $(\mu^N)_N$ by contraction principle. 

We first introduce the rate function in Section \ref{ratefunctionsection}. Then, because of the discrete character of the second type particles, we introduce in Section \ref{discretness} a convenient closed subspace of $\M_{1/2}(\R_-)$ where the $\nu^N$'s actually live. Finally, we state the LDP for $(\mu^N,\nu^N)_N$ in Section \ref{LDPgeneralsection}, see Theorem \ref{generalLDP}, and provide a proof for Theorem \ref{LDP}.  The proof of Theorem \ref{generalLDP} is deferred to Section \ref{proof}.

\subsection{The rate function}
\label{ratefunctionsection}
Our first task is to extend properly the definition of the functional \eqref{ratefunctionfake} to $\M_1(\R_+)\times\M_{1/2}(\R_-)$. A general method to do so has been presented in \cite{HK} and is based on a compactification procedure that we present for our particular case now.

\subsubsection{Compactification procedure}
Let $\s$ be the circle of $\R^2$ centered in $(0,1/2)$ of radius $1/2$ and $T:\R\rightarrow \s$  the associated inverse stereographic projection, namely the map defined by 
\[
T(x)= \left(\frac{x}{1+x^2}, \frac{x^2}{1+x^2}\right), \qquad x\in\R.
\] 
It is known that $T$ is an homeomorphism from $\R$ onto $\s\setminus\{(0,1)\}$, so that $(\s,T)$ is a one point compactification of $\R$. For a measure $\mu$ on $\R$, we  denote by $T_*\mu$  its push-forward by $T$, that is the measure on $\s$ characterized by
\eq
\label{pushf}
\int_{\s} f(x)dT_*\mu(x) = \int_\R f \big(T(x)\big)d\mu(x)
\qe
for every Borel function $f$ on $\s$.  
We denote  the two half-circles 
\eq
\label{halfcircle}
\s_\pm=\Big\{T(x):\; x\in\R_\pm\Big\}\cup\big\{(0,1)\big\}.
\qe
Since $T$ is an homeomorphism from $\R_+$ (resp. $\R_-$) to $\s_+\setminus\{(0,1)\}$ (resp. $\s_-\setminus\{(0,1)\})$, it is easy to see (cf. \cite[Lemma 2.1]{H}) that $T_*$ is a homeomorphism from $\M_1(\R_+)$ to   
\[
\big\{\mu\in\M_{1}(\s_+):\;\mu(\{(0,1)\})=0\big\},
\]
and also from $\M_{1/2}(\R_-)$ to 
\[
\big\{\mu\in\M_{1/2}(\s_-):\;\mu(\{(0,1)\})=0\big\}.
\]

Equipped with such a transformation $T_*$, we are now able to provide a proper definition for the functional \eqref{ratefunctionfake}.

\subsubsection{Definition of the rate function}

Introduce the lower semi-continuous function $\bs\V:\s_+\rightarrow\R\cup\{+\infty\}$ by

\eq
\label{curlyV1}
\bs\V\big(T(x)\big)=x-2\sqrt{ax}-\frac{3}{4}\log(1+x^2), \qquad x\in\R_+,
\qe
and 
\eq
\label{curlyV2}
\bs\V((0,1))=\liminf_{x\rightarrow\infty}\bs\V\big(T(x)\big)=+\infty.
\qe 
We naturally extend the definition of the logarithmic energy \eqref{logenergy} to measures on $\s\subset\R^2$ (where $| \cdot |$ stands for the Euclidean norm) and define the functional $\J$ on $\M_1(\R_+)\times\M_{1/2}(\R_-)$ by

\begin{align}
\label{ratefunction}
\J(\mu,\nu)= & \iint \log\frac{1}{|z-w|}dT_*\mu(z)dT_*\mu(w) - \iint \log\frac{1}{|z-\xi|}dT_*\mu(z)dT_*\nu(\xi) \nonumber \\
 & \qquad+ \iint \log\frac{1}{|\xi-\zeta|}dT_*\nu(\xi)dT_*\nu(\zeta) +\int \bs \V(z)dT_*\mu(z)
\end{align}
when both $T_*\mu$ and $T_*\nu$ have finite logarithmic energy, and set $J(\mu,\nu)=+\infty$ otherwise. 

This definition is motivated by the following observation : from the metric relation \cite[Lemma 3.4.2]{Ash},
\begin{equation}
\label{metric}
|T(x)-T(y)|=\frac{|x-y|}{\sqrt{1+x^2}\sqrt{1+y^2}}, \qquad x, y \in \R,
\end{equation}
we  obtain  with \eqref{pushf}  for any Borel measures $\mu$, $\nu$ on $\R$  the relation

\begin{align}
\label{relationCsphere}
  \iint \log\frac{1}{|z-w|}dT_*\mu(z)dT_*\nu(w) & = \iint \log\frac{1}{|x-y|}d\mu(x)d\nu(y)\nonumber\\
& \qquad +\, \frac{1}{2}\mu(\R)\int\log(1+x^2)d\nu(x)\\
&  \qquad  +\frac{1}{2}\nu(\R)\int\log(1+x^2)d\mu(x)\nonumber,
\end{align}
as soon as one assumes all these quantities to be finite. In this case, we obtain that $\J(\mu,\nu)$ matches with \eqref{ratefunctionfake}. 

Then, the following Proposition is a consequence of  \cite[Theorem 2.6]{HK}.

\begin{proposition}
\label{GRF}\
\begin{itemize}
\item[{\rm (a)}]
The level set
\[
\Big\{(\mu,\nu)\in\M_1(\R_+)\times\M_{1/2}(\R_-) : \;\J(\mu,\nu)\leq \gamma\Big\}
\]
is compact for all $\gamma\in\R$. 
\item[{\rm (b)}]
$\J$ is strictly convex on the set where it is finite.
\end{itemize}
\end{proposition}

Because of the discrete character of the $u_i$'s, we need to discuss now several constraint issues.

\subsection{Discreteness and constraint}

In this section we use  the discrete character of the particles on $\R_-$ to build a closed  subset $\E(\R_-)$ of $\M_{1/2}(\R_-)$ such that $\nu^N\in\E(\R_-)$ for all $N$. This will provide an explanation on why the measures on $\R_-$ are restricted to the set $\M_{1/2}^\sigma(\R_-)$ in the minimization problem \eqref{ratefunctionfake}, and moreover will be of important use to control the possible contact at the origin of the different type particles during the proof of Theorem \ref{generalLDP}, see Lemma \ref{unifint}.

\label{discretness}
We say that a measure $\nu\in\Mb$ is \emph{constrained} by a Borel measure $\lambda$ on $\R_-$, that we note $\nu\leq \lambda$, if the signed measure $\lambda-\nu$ is in a fact a (positive) measure. Introduce the set of constrained measures
 
\eq
\M_{1/2}^\lambda(\R_-)=\Big\{\nu\in\Mb :\;\nu\leq\lambda\Big\}
\qe
and note it is closed. Indeed, if $(\nu_N)_N$ is a sequence in $\M_{1/2}^\lambda(\R_-)$ with weak limit $\nu$, then $(\lambda-\nu_N)_N$ converges in the vague topology (i.e the topology coming from duality with the Banach space of compactly supported continuous functions on $\R$) towards $\lambda-\nu$, which is hence not signed. 

Since the random variables $u_i$'s take values in $\A_N$, see \eqref{AN}, we have almost surely
\[
\nu^N=\frac{1}{N}\sum_{i=1}^{N/2}\delta(u_i)\leq \frac{1}{N}\sum_{u\in\A_N}\delta(u)=\sigma_N
\] 
and thus almost surely $\nu^N\in\Mf$ for any $N$.  Consider the measure $\sigma$ on $\R_-$ having for density
\eq
\label{sigma}
\frac{d\sigma(x)}{dx} = \frac{\sqrt{a}}{\pi}|x|^{-1/2}.
\qe
and note that the Radon-Nikodym theorem yields that the definition of $\M_{1/2}^\sigma(\R_-)$ presented in this section matches with \eqref{constrspace}.  It is in fact the limiting distribution of the constraints $\sigma_N$.

\begin{lemma}
\label{sigmaconv}
The sequence $(\sigma_N)_N$ converges towards $\sigma$ in the vague topology.
\end{lemma}

\begin{proof}
Since for any $b\leq0$ we clearly have $\lim_{N\rightarrow\infty}\sigma_N(\{b\})=\sigma(\{b\})=0$,
it is enough to show that for all $b<0$, $\lim_{N\rightarrow\infty}\sigma_N([b,0])=\sigma([b,0])$,
that is 
\[
\lim_{N\rightarrow\infty}\frac{1}{N}\sharp \Big\{k:\,a_{k,N}\geq b\Big\}=\frac{\sqrt{a}}{\pi}\int_{b}^0|x|^{-1/2}dx.
\]
By change of variables, it is equivalent to prove that for all $b>0$
\[
\lim_{N\rightarrow \infty}\frac{1}{N}\sharp \Big\{k:\,\frac{j_{\alpha,k}}{N}\leq b\Big\} = \frac{1}{\pi}\int_{0}^bdx=\frac{b}{\pi}.
\]
Fix $\varepsilon>0$ and let $k(N)$  be the integer part of $\frac{(b+\varepsilon)N}{\pi}$. The McMahon asymptotic formula  \cite[formula 9.5.12]{AS}  yields  
\eq
\label{McMahon1}
\lim_{k\rightarrow\infty}\frac{j_{\alpha,k}}{k}=\pi,
\qe
and thus
\[
\lim_{N\rightarrow\infty}\frac{j_{\alpha,k(N)}}{N}=b+\varepsilon.
\]
 As a consequence, we obtain the upper bound
\[
\limsup_{N\rightarrow \infty}\frac{1}{N}\sharp \Big\{k:\,\frac{j_{\alpha,k}}{N}\leq b\Big\}\leq \lim_{N\rightarrow \infty}\frac{k(N)}{N}=\frac{b+\varepsilon}{\pi}.
\]
Similarly, changing $\varepsilon$ by $-\varepsilon$ in the definition of $k(N)$ yields the lower bound
\[
\liminf_{N\rightarrow \infty}\frac{1}{N}\sharp \Big\{k:\,\frac{j_{\alpha,k}}{N}\leq b\Big\}\geq \frac{b-\varepsilon}{\pi},
\] 
and Lemma \ref{sigmaconv} follows by letting $\varepsilon\rightarrow 0$.
\end{proof}

We now introduce the subset $\Me$ of $\Mb$ of  the measures which are either constrained by $\sigma_N$, for some $N$, or by $\sigma$. Namely,

\eq
\label{defER}
\Me = \bigcup_{N=1}^\infty \Mf \bigcup \Mc.
\qe
By construction $\nu^N\in\Me$, for any $N$, and moreover

\begin{lemma}
\label{closed}
$\Me$ is a closed subset of $\Mb$.
\end{lemma}
\begin{proof}
Let $(\nu_j)_j$ be a sequence in $\Me$ with weak limit $\nu$, and let us show that $\nu\in\E(\R_-)$. Since the sets $\Mc$ and $\Mf$ are closed for all $N$, one may assume that $\nu_j\leq \sigma_{N_j}$, with $\lim_{j\rightarrow\infty}N_j=+\infty$. One then obtains by Lemma \ref{sigmaconv} that $\nu\leq \sigma$, and thus $\nu\in\Me$.
\end{proof} 
Concerning the measure on $\s_-$, see \eqref{halfcircle}, we similarly set
\eq
\label{defES}
\E(\s_-)=\bigcup_{N=1}^\infty \M_{1/2}^{T_*\sigma_N}(\s_-) \bigcup \M_{1/2}^{T_*\sigma}(\s_-),
\qe
so that $T_*\nu^N\in\E(\s_-)$ for any $N$. Moreover, note that since $\nu(\{(0,1)\})=0$ for any $\nu\in\E(\s_-)$, it follows that $T_*$ is an homeomorphism from $\E(\R_-)$ to $\E(\s_-)$, and $\E(\s_-)$ is seen to be a closed subset of $\M_{1/2}(\s_-)$ from Lemma \ref{closed}.

\subsection{LDP for the generalized particle system}
\label{LDPgeneralsection}
We are now in  a position to state the  LDP for $(\mu^N,\nu^N)_N$. Let us precise that we equip $\E(\R_-)$ with the topology induced by $\M_{1/2}(\R_-)$ and $\M_1(\R_+)\times\E(\R_-)$ carries the product one. Then the following LDP holds.

\begin{theorem}
\label{generalLDP}
The sequence $(\mu^N,\nu^N)_N$ satisfies a LDP on $\M_1(\R_+)\times\E(\R_-)$ in the scale $N^2$ with good rate function $\J-\min\J$. More precisely, 
\begin{itemize}
\item[{\rm (a)}] The level set 
\[
\Big\{ (\mu,\nu)\in\M_1(\R_+)\times\E(\R_-) : \; \J(\mu,\nu)\leq \gamma\Big\}
\]
is compact for any $\gamma\in\R$.

\item[{\rm (b)}] $\J$ admits a unique minimizer $(\mu^*,\nu^*)$ on $\M_1(\R_+)\times\E(\R_-)$.

\item[{\rm (c)}] For any closed set $\F\subset \M_1(\R_+)\times\E(\R_-)$,
\[
\limsup_{N\rightarrow\infty}\frac{1}{N^2}\log \p_N\Big((\mu^N,\nu^N)\in\F\Big)\leq - \inf_{(\mu,\nu)\in\F}\Big\{\J(\mu,\nu)-\J(\mu^*,\nu^*)\Big\}.
\]
\item[{\rm (d)}] For any open set $\mathcal O\subset \M_1(\R_+)\times\E(\R_-)$,
\[
\liminf_{N\rightarrow\infty}\frac{1}{N^2}\log \p_N\Big((\mu^N,\nu^N)\in\mathcal O\Big)\geq - \inf_{(\mu,\nu)\in\mathcal O}\Big\{\J(\mu,\nu)-\J(\mu^*,\nu^*)\Big\}.
\]
\end{itemize}
\end{theorem}

A direct consequence of Theorem \ref{generalLDP} is Theorem \ref{LDP}.

\begin{corollary}
\label{contractionprinciple} 
Theorem \ref{LDP} holds true.
\end{corollary}

\begin{proof}
Theorem \ref{LDP} follows by  contraction principle (see \cite[Theorem 4.2.1]{DZ}) along the projection $\Ma\times\E(\R_-)\rightarrow\Ma$ and the fact that $\J(\mu,\nu)=+\infty$ as soon as $\nu\in\E(\R_-)\setminus\M_{1/2}^\sigma(\R_-)$.
\end{proof}

\section{Proof of Theorem \ref{generalLDP}}
\label{proof}

We first observe that Theorem \ref{generalLDP} (a), (b) easily follow from Proposition \ref{GRF}. 

\begin{proof}[Proof of Theorem \ref{generalLDP} {\rm (a)}, {\rm (b)}.]
Since $\E(\R_-)$ is a  closed subset of $\M_{1/2}(\R_-)$ (see Lemma \ref{closed}), Theorem \ref{generalLDP} (a) follows from Proposition \ref{GRF} (a). The existence of a minimizer for $\J$ on $\M_{1/2}(\R_-)\times\E(\R_-)$ is a consequence of Theorem \ref{generalLDP} (a). Since the set  $\M_{1/2}^\sigma(\R_-)$ is convex, and $\J(\mu,\nu)=+\infty$ as soon as $\nu\in\E(\R_-)\setminus\M_{1/2}^\sigma(\R_-)$, the minimizer is unique by Proposition \ref{GRF} (b).
\end{proof}

Concerning the proof of Theorem \ref{generalLDP} (c), (d), it is usually pretty standard to establish LDP upper and lower bounds by proving a weak LDP and an exponential tightness property (see \cite{DZ} for a general presentation on LDPs). However, because of the lack of confining potential acting on the particles on $\R_-$, it is not clear to the authors how to prove directly  that the sequence $(\mu^N,\nu^N)_N$ is exponentially tight. Instead, we  follow a different strategy  developed in \cite{H} : we  first prove in Section \ref{LDPUB} a weak LDP upper bound for $(T_*\mu^N,T_*\nu^N)_N$, the push-forward of $(\mu^N,\nu^N)_N$ by the inverse stereographic projection $T$. We then establish a LDP lower bound for $(\mu^N,\nu^N)_N$ in Section \ref{LDPLB}, and show in Section \ref{fullLDP} that it is enough to obtain Theorem \ref{generalLDP} (c), (d).  

\subsection{A weak LDP upper bound for $(T_*\mu^N,T_*\nu^N)_N$}
Consider the functional $J$ on $\M_1(\s_+)\times\E(\s_-)$ defined by
\begin{align}
\label{J}
J(\mu,\nu)= & \iint \log\frac{1}{|z-w|}d\mu(z)d\mu(w) - \iint \log\frac{1}{|z-\xi|}d\mu(z)d\nu(\xi) \nonumber \\
 & \qquad+ \iint \log\frac{1}{|\xi-\zeta|}d\nu(\xi)d\nu(\zeta) +\int \bs \V(z)d\mu(z)
\end{align}
if both $\mu$ and $\nu$ have finite logarithmic energy, and set $J(\mu,\nu)=+\infty$ otherwise. We recall that $\bs\V$ has been introduced in \eqref{curlyV1}--\eqref{curlyV2} and $\E(\s_-)$ in \eqref{defES}. Note that, with $\J$ defined in \eqref{ratefunction},  the following relation holds
\eq
\label{relation}
\J(\mu,\nu)=J(T_*\mu,T_*\nu), \qquad (\mu,\nu)\in\M_1(\R_+)\times\E(\R_-).
\qe

Now, choose a metric compatible with the topology of $\M_1(\s_+)\times\E(\s_-)$ and write $\B_{\delta}(\mu,\nu)$ for the open ball of radius $\delta>0$ centered at $(\mu,\nu)$. The aim of this section is to establish the following weak LDP upper bound for $(T_*\nu^N,T_*\nu^N)$.

\begin{proposition}
\label{LDPUB}
For any $(\mu,\nu)\in\M_1(\s_+)\times \E(\s_-)$
\eq
\label{LDPUBeq}
\limsup_{\delta\rightarrow 0}\limsup_{N\rightarrow\infty}\frac{1}{N^{2}}\log\Big\{ Z_N\p_N\Big(\left(T_*\mu^{N},T_*\nu^{N}\right)\in \B_{\delta}(\mu,\nu)\Big)\Big\} \leq-J(\mu,\nu).
\qe
\end{proposition} 

Concerning the proof, we first describe in Section \ref{RVonsphere} the induced  distribution  for the particles $\big(T(x_i)\big)_{i=1}^N$ and $\big(T(u_i)\big)_{i=1}^{N/2}$ on $\s_+^N\times\s_-^{N/2}$. Then, we show \eqref{LDPUBeq} in Section \ref{coreofproof}, where the main difficulty is to control the singularity created by the fact that the different type particles may meet at the origin when $N\rightarrow\infty$.  To do so, we will use a  few technical lemmas, for which the proofs are deferred to Section \ref{lemmas} for  convenience.

\subsubsection{The induced distribution for the particles on $\s$}
\label{RVonsphere}

Introduce the random variables  on $\s$
\eq
z_i=T(x_i), \quad i=1,\ldots,N,\qquad \xi_i=T(u_i), \quad i=1,\ldots, N/2,
\qe
where the $x_i$'s and the $u_i$'s are distributed according to \eqref{generaldistr}. Thus
\eq
\label{ESDsphere}
T_*\mu^N=\frac{1}{N}\sum_{i=1}^N\delta(z_i),\qquad\qquad T_*\nu^N=\frac{1}{N}\sum_{i=1}^{N/2}\delta(\xi_i). 
\qe
We set the measures  $\lambda=T_*(\bs 1_{\R_+}(x)dx)$ on $\s_+$ and $\eta_N=T_*\sigma_N$  on $\s_-$,  with $\sigma_N$  introduced in \eqref{sigmaN}. From $V_N$  introduced in \eqref{VN}, we also construct the lower semi-continuous function  $\bs\V_N : \s_+\rightarrow\R\cup\{+\infty\}$ by 
\eq
\label{curlyVN1}
\bs\V_N\big(T(x)\big)=V_N(x) - \frac{3}{4}\log(1+x^2), \qquad x\in\R_+,
\qe
and 
\eq
\label{curlyVN2}
\bs\V_N((0,1))=\liminf_{x\rightarrow\infty}\bs\V_N\big(T(x)\big)=+\infty,
\qe
where the latter equality follows from the asymptotic behavior  \cite[formula 9.7.1]{AS}
\eq
\label{asymptoBessel}
I_{\alpha}(x)=\frac{e^{x}}{\sqrt{2\pi x}}\Big(1+O\left(x^{-1}\right)\Big) \qquad \mbox{as   } x\rightarrow +\infty.
\qe 
Then the following  holds.

\begin{lemma}
\label{generaldistr2}
The joint distribution of  $(\bs z,\bs\xi)=(z_1,\ldots,z_N,\xi_1,\ldots,\xi_{N/2})$ is given by 
\[
\frac{1}{Z_N}\left|\frac{\Delta_N^2(\bs z)\Delta_{N/2}^2( \bs \xi)}{\Delta_{N,N/2}(\bs z, \bs \xi)}\right|\prod_{i=1}^N (1-|z_i|^2)e^{-N\bs \V_N(z_i)}d\lambda(z_i)\prod_{i=1}^{N/2}|\xi_i|\sqrt{1-|\xi_i|^2}d\eta_N(\xi_i)
\]
where $Z_N$ has been introduced in \eqref{generaldistr}.
\end{lemma}

\begin{proof}
From the metric relation \eqref{metric} we obtain  
\begin{align*}
\Delta^2_N(\bs x)  & = \big|\Delta^2_N\big(T(\bs x)\big)\big| \prod_{i=1}^N \big(1+x_i^2\,\big)^{N-1} \\
\Delta^2_{N/2}(\bs u)  & = \big|\Delta^2_{N/2}\big(T(\bs u)\big)\big| \prod_{i=1}^{N/2} \big(1+u_i^2\,\big)^{N/2-1}\\
\Delta_{N,N/2}(\bs x,\bs u) & = \big|\Delta_{N,N/2}\big(T(\bs x),T(\bs u)\big)\big|\prod_{i=1}^N \big(1+x_i^2\,\big)^{N/4}\prod_{i=1}^{N/2} \big(1+u_i^2\,\big)^{N/2}.
\end{align*}
Thus, with $\bs \V_N$ defined in \eqref{curlyVN1}, this yields
\begin{multline}
\label{S1}
\frac{\Delta_N^2(\bs x)\Delta_{N/2}^2( \bs u)}{\Delta_{N,N/2}(\bs x, \bs u)}\prod_{i=1}^N e^{-N V_N(x_i)}\prod_{i=1}^{N/2}|u_i|\\
 = \left|\frac{\Delta_N^2\big(T(\bs x)\big)\Delta_{N/2}^2\big(T( \bs u)\big)}{\Delta_{N,N/2}\big( T(\bs x), T(\bs u)\big)}\right|\prod_{i=1}^N \frac{e^{-N\bs \V_N(T(x_i))}}{1+x_i^2}\prod_{i=1}^{N/2} \frac{|u_i|}{1+u_i^2}.
\end{multline}
We moreover obtain  from \eqref{metric} the identities 
\eq
\label{metric2}
\frac{1}{1+x^2}=1-|T(x)|^2, \quad \frac{|x|}{1+x^2}=|T(x)|\sqrt{1-|T(x)|^2},\qquad x\in\R,
\qe
and then from \eqref{S1} 
\begin{multline}
\label{S2}
\frac{\Delta_N^2(\bs x)\Delta_{N/2}^2( \bs u)}{\Delta_{N,N/2}(\bs x, \bs u)}\prod_{i=1}^N e^{-N V_N(x_i)}\prod_{i=1}^{N/2}|u_i|\\
 = \left|\frac{\Delta_N^2\big(T(\bs x)\big)\Delta_{N/2}^2\big(T( \bs u)\big)}{\Delta_{N,N/2}\big( T(\bs x), T(\bs u)\big)}\right|\prod_{i=1}^N \big(1-|T(x_i)|^2\big)e^{-N\bs \V_N(T(x_i))^2}\prod_{i=1}^{N/2}|T(u_i)|\sqrt{1-|T(u_i)|^2}.
\end{multline}
Lemma \ref{generaldistr2} then follows from \eqref{S2} by performing the change of variables $z_i=T(x_i)$ for $i=1,\ldots,N$ and $\xi_i=T(u_i)$ for $i=1,\ldots,N/2$. 
\end{proof}

\subsubsection{Core of the proof for Proposition \ref{LDPUB}}
\label{coreofproof}

Provided with Lemma \ref{generaldistr2}, we now establish  Proposition \ref{LDPUB}, up to the proofs of few  lemmas which are deferred to the next section.

\begin{proof}[Proof of Proposition \ref{LDPUB}]
We obtain from \eqref{ESDsphere} and Lemma \ref{generaldistr2}
\begin{multline}
\label{A1}
Z_N\p_N\Big((T_*\mu^N,T_*\nu^N)\in\B_\delta(\mu,\nu)\Big)\\
=   \int_{\big\{ (\bs z , \bs\xi)\,:\,(T_*\mu^N, T_*\nu^N)\in\B_\delta(\mu,\nu)\big\}}\left|\frac{\Delta_N^2(\bs z)\Delta_{N/2}^2( \bs \xi)}{\Delta_{N,N/2}(\bs z, \bs \xi)}\right|\prod_{i=1}^N e^{-N\bs \V_N(z_i)}\\
\times\prod_{i=1}^N(1-|z_i|^2)d\lambda(z_i)\prod_{i=1}^{N/2}|\xi_i|\sqrt{1-|\xi_i|^2}d\eta_N(\xi_i).
\end{multline}
We write
\begin{align}
\label{A2}
  &\left|\frac{\Delta_N^2(\bs z)\Delta_{N/2}^2( \bs \xi)}{\Delta_{N,N/2}(\bs z, \bs \xi)}\right|\prod_{i=1}^N e^{-N\bs \V_N(z_i)} \nonumber \\
= &  \quad 
\exp \Bigg(- \Bigg\{ \sum_{1\leq i \neq j \leq N}\log\frac{1}{|z_i-z_j|} +  \sum_{1\leq i \neq j \leq N/2}\log\frac{1}{|\xi_i-\xi_j| } \nonumber \\
   &\qquad \qquad \qquad \qquad \qquad\qquad+ \sum_{i=1}^N\sum_{j=1}^{N/2}\Big( 2\bs \V_N(z_i) +\log|z_i-\xi_j| \Big) \Bigg\}\Bigg)\nonumber \\
  \quad 
  = & \quad \exp \Bigg(- N^2\Bigg\{\iint _{z\neq w}\log\frac{1}{|z-w|}dT_*\mu^N(z)dT_*\mu^N(w)\\
 &  \qquad \qquad \qquad\qquad +  \iint _{\xi\neq \zeta}\log\frac{1}{|\xi-\zeta|}dT_*\nu^N(\xi)dT_*\nu^N(\zeta) \nonumber \\
 &  \qquad \qquad \qquad \qquad \qquad+ \iint \Big( 2 \bs\V_N(z) +\log|z-\xi| \Big)dT_*\mu^N(z)dT_*\nu^N(\xi) \Bigg\}\Bigg).\nonumber
\end{align}
Note that, since $T_*\mu^N\otimes T_*\mu^N\big\{(z,w)\in \s_+\times\s_+ :\,z=w\big\}=1/N$ almost surely, for any $M>0$ we have almost surely
\begin{align}
\label{A3}
& \iint _{z\neq w}\log\frac{1}{|z-w|}dT_*\mu^N(z)dT_*\mu^N(w) \nonumber \\
\geq  & \quad \iint \min\Big( \log\frac{1}{|z-w|}, M\Big)dT_*\mu^N(z)dT_*\mu^N(w) -\frac{M}{N}
\end{align}
and similarly 
\begin{align}
\label{A4}
& \iint _{\xi\neq \zeta}\log\frac{1}{|\xi-\zeta|}dT_*\nu^N(\xi)dT_*\nu^N(\zeta) \nonumber \\
\geq  & \quad \iint \min\Big( \log\frac{1}{|\xi-\zeta|}, M\Big)dT_*\nu^N(\xi)dT_*\nu^N(\zeta) -\frac{M}{2N}.
\end{align}
To make  the control of the singularity at the origin easier, we write  for  any $M>0$
\begin{align}
\label{A5}
&\iint \Big( 2 \bs\V_N(z) +\log|z-\xi| \Big)dT_*\mu^N(z)dT_*\nu^N(\xi)\nonumber\\
=& \; \iint \Big( 2 \bs\V_N(z) +\log|z-\xi|-\log|\xi| \Big)dT_*\mu^N(z)dT_*\nu^N(\xi)\nonumber \\
 & \qquad\qquad+\int\log|\xi|dT_*\nu^N(\xi)\nonumber\\
\geq &  \;\iint \min\Big( 2 \bs\V_N(z) +\log|z-\xi|-\log|\xi| , M\Big)dT_*\mu^N(z)dT_*\nu^N(\xi)\\
& \qquad\qquad +\int\log|\xi|dT_*\nu^N(\xi).\nonumber 
\end{align}
Note that the latter step makes sense since $T_*\nu^N$ can not  have a mass point at $(0,0)$. Such a decomposition is motivated by the following lemma. 
\begin{lemma}\
\label{boundedM}
For any $N\in\N\cup\{\infty\}$, the map
\eq
\label{nicemap}
(z,\xi)\mapsto 2 \bs\V_N(z) +\log|z-\xi|-\log|\xi|
\qe
is bounded from below on $\s_+\times\s_-$, where we denote $\bs\V_\infty=\bs\V$.
\end{lemma}

Now, if we introduce for any $M>0$ and $(\mu,\nu)\in\M_1(\s_+)\times\E(\s_-)$ 
\begin{align}
\label{JMN}
J^M_N(\mu,\nu) & =  \iint \min\Big( \log\frac{1}{|z-w|}, M\Big)d\mu(x)d\mu(y)\\
 & \qquad + \iint \min \Big(2 \bs\V_N(z) +\log|z-\xi|-\log|\xi|, M\Big)d\mu(z)d\nu(\xi)\nonumber \\
 & \qquad \qquad  +\iint \min\Big( \log\frac{1}{|\xi-\zeta|}, M\Big)d\nu(\xi)d\nu(\zeta) + \int\log|\xi|d\nu(\xi)\nonumber,
\end{align}  
we obtain from \eqref{A1}--\eqref{A5} that
\eq
\label{A6}
Z_N\p_N\Big((T_*\mu^N,T_*\nu^N)\in\B_\delta(\mu,\nu)\Big)  \leq C_N\exp\Big\{ - N^2 \inf_{\B_\delta(\mu,\nu)} J_N^M\Big\},
\qe
where we set 
\[
\label{CN}
C_N =e^{3MN/2}\int_{\s_+^N\times \s_-^{N/2}}\prod_{i=1}^N(1-|z_i|^2)d\lambda(z_i)\prod_{i=1}^{N/2}|\xi_i|\sqrt{1-|\xi_i|^2}d\eta_N(\xi_i).
\]
Note that by construction $J^M_N$ is bounded from above, but may take the value $-\infty$ for some $(\mu,\nu)\in\M_1(\s_+)\times\M_{1/2}(\s_-)$. Our choice to restrict $\M_{1/2}(\s_-)$  to $\E(\s_-)$ is motivated by the following key lemma, which yields in particular that $J^M_N$ is well defined and has each of its components bounded on $\M_1(\s_+)\times\E(\s_-)$.

\begin{lemma}
\label{unifint}
The functional 
\[
\nu\mapsto \int\log|\xi|d\nu(\xi)
\]
is continuous, and thus bounded, on $\E(\s_-)$.
\end{lemma}

We observe that 
\begin{lemma}
\label{CNlim}
\eq
\label{CNlimeq}
\limsup_{N\rightarrow\infty}\frac{1}{N^2}\log C_N\leq 0.
\qe
\end{lemma}

As a consequence, we  obtain from \eqref{A6} 
\eq
\label{A7}
\limsup_{N\rightarrow\infty}\frac{1}{N^{2}}\log \Big\{Z_N\p\Big(\left(T_*\mu^{N},T_*\nu^{N}\right)\in \B_{\delta}(\mu,\nu)\Big)\Big\} \leq - \liminf_{N\rightarrow\infty}\inf_{\B_\delta(\mu,\nu)} J_N^M.
\qe
Now,  introduce for any $M>0$ and $(\mu,\nu)\in\M_1(\s_+)\times\E(\s_-)$ 
\begin{align}
\label{JM}
J^M(\mu,\nu) & =  \iint \min\Big( \log\frac{1}{|z-w|}, M\Big)d\mu(x)d\mu(y)\\
 & \qquad + \iint \min \Big(2 \bs\V(z) +\log|z-\xi|-\log|\xi|, M\Big)d\mu(z)d\nu(\xi)\nonumber \\
 & \qquad \qquad  +\iint \min\Big( \log\frac{1}{|\xi-\zeta|}, M\Big)d\nu(\xi)d\nu(\zeta) + \int\log|\xi|d\nu(\xi)\nonumber
\end{align}
since the following holds
\begin{lemma}
\label{liminfineq}
\eq
\liminf_{N\rightarrow\infty}\inf_{\B_\delta(\mu,\nu)} J_N^M \geq \inf_{\B_\delta(\mu,\nu)} J^M.
\qe
\end{lemma}

It thus follows from \eqref{A7} that
\eq
\label{A9}
\limsup_{N\rightarrow\infty}\frac{1}{N^{2}}\log \Big\{Z_N\p\Big(\left(T_*\mu^{N},T_*\nu^{N}\right)\in \B_{\delta}(\mu,\nu)\Big)\Big\} \leq - \inf_{\B_\delta(\mu,\nu)} J^M.
\qe
Note that for any $M>0$, the function
\[
(z,w)\mapsto \min\Big( \log\frac{1}{|z-w|}, M\Big)
\]
is continuous on $\s\times\s$, so that the functional 
\[
\mu \mapsto \iint \min\Big( \log\frac{1}{|z-w|}, M\Big)d\mu(z)d\mu(w)
\]
is continuous on $\M_1(\s_+)$, as well on $\E(\s_-)$.  Lemma \ref{boundedM} moreover yields for any $M>0$ the continuity of 
\[
(\mu,\nu)\mapsto  \iint \min \Big(2 \bs\V(z) +\log|z-\xi|-\log|\xi|, M\Big)d\mu(z)d\nu(\xi).
\]
Thus, this shows  with Lemma \ref{unifint} that $J^M$ defined in \eqref{JM} is continuous on $\M_{1}(\s_+)\times\E(\R_-)$, and  we obtain by letting $\delta\rightarrow 0$ in \eqref{A9} that
\eq
\label{A10}
 \limsup_{\delta\rightarrow0}\limsup_{N\rightarrow\infty}\frac{1}{N^{2}}\log \Big\{Z_N\p\Big(\left(T_*\mu^{N},T_*\nu^{N}\right)\in \B_{\delta}(\mu,\nu)\Big)\Big\}\leq - J^M(\mu,\nu).
\qe
Letting $M\rightarrow+\infty$ in \eqref{A10}, the monotone convergence theorem yields 
\begin{align}
\label{A12}
 & \limsup_{\delta\rightarrow0}\limsup_{N\rightarrow\infty}\frac{1}{N^{2}}\log \Big\{Z_N\p\Big(\left(T_*\mu^{N},T_*\nu^{N}\right)\in \B_{\delta}(\mu,\nu)\Big)\Big\}\nonumber \\
 \leq &\; - \Bigg\{\iint\log\frac{1}{|z-w|}d\mu(z)d\mu(w) \\
 & \qquad \quad +\iint \Big(2\bs\V(z)+\log|z-\xi|-\log|\xi|\Big)d\mu(z)d\nu(\xi)\nonumber\\
  & \qquad \quad \qquad +\iint\log\frac{1}{|\xi-\zeta|}d\nu(\xi)d\nu(\zeta) +\int\log|\xi|d\nu(\xi)\Bigg\}.\nonumber
\end{align}
Finally, in order to obtain Proposition \ref{LDPUB} from \eqref{A12}, it is  sufficient to show that, with $J$ defined in \eqref{J},
\begin{align}
\label{representation}
J(\mu,\nu)= & \iint\log\frac{1}{|z-w|}d\mu(z)d\mu(w)\\
& \quad  +\iint \Big(2\bs\V(z)+\log|z-\xi|-\log|\xi|\Big)d\mu(z)d\nu(\xi) \nonumber\\
 & \qquad\qquad +\iint\log\frac{1}{|\xi-\zeta|}d\nu(\xi)d\nu(\zeta)+\int\log|\xi|d\nu(\xi)\nonumber
\end{align}
for all $(\mu,\nu)\in\M_1(\s_+)\times\E(\s_-)$.  Note that if $\mu$ or $\nu$ has infinite logarithmic energy, then  Lemmas \ref{boundedM} and \ref{unifint} yield that the right-hand side of \eqref{representation} is $+\infty$.  If both $\mu$ and $\nu$ have finite logarithmic energy, then it is known  (see e.g. \cite[Section 3.1]{HK}) that
\[
\iint\log\frac{1}{|z-\xi|}d\mu(z)d\nu(\xi)<+\infty,
\]
and thus since $\V$ is bounded from below
\begin{align*}
 & \iint \Big(2\bs\V(z)+\log|z-\xi|-\log|\xi|\Big)d\mu(z)d\nu(\xi)+\int\log|\xi|d\nu(\xi)\\
=& \quad \int \bs\V(z)d\mu(z) -\iint\log\frac{1}{|z-\xi|}d\mu(z)d\nu(\xi),
\end{align*}
which proves \eqref{representation}. The proof of Proposition \ref{LDPUB} is therefore complete, up to the proofs of the lemmas.

\end{proof}

\subsubsection{Proofs of  Lemmas \ref{boundedM}, \ref{unifint}, \ref{CNlim}, and \ref{liminfineq}}
\label{lemmas}

\begin{proof}[Proof of Lemma \ref{boundedM}]
We have the inequality
\eq
\label{ineqsphere}
|z-\xi|\geq |\xi|\sqrt{1-|z|^2}, \qquad z\in\s_+,\quad \xi\in\s_-.
\qe
Indeed, \eqref{ineqsphere} trivially holds if $z=(0,1)$. Since for any $z\in\s$  the Pythagorean theorem yields $|z-(0,1)|=\sqrt{1-|z|^2}$,  \eqref{ineqsphere}  moreover holds when $\xi=(0,1)$. If none of $z$ or $\xi$ is $(0,1)$, then  there exist $x\in\R_+$ and $u\in\R_-$ such that $|z-\xi|=|T(x)-T(u)|$. Inequality \eqref{ineqsphere} then follows from the metric relations \eqref{metric}, \eqref{metric2} and the inequality $|x-u|\geq |u|$ when $(x,u)\in\R_+\times\R_-$. 

As a consequence of the inequality  \eqref{ineqsphere}, we obtain for any $(z,\xi)\in\s_+\times\s_-$ and $N\in\N\cup\{\infty\}$
\eq
\label{D1}
 2 \bs\V_N(z) +\log|z-\xi|-\log|\xi|\geq 2 \bs\V_N(z) + \frac{1}{2}\log(1-|z|^2).  
\qe
Now, from the the metric relations \eqref{metric2} we obtain
\eq
\label{D2}
\inf_{z\in\s_+}\Big(2\bs\V_\infty(z) + \frac{1}{2}\log(1-|z|^2)\Big)
= 2 \inf_{x\in \R_+}\Big(x-2\sqrt{ax}-\log(1+x^2)\Big)>-\infty,
\qe
and similarly for any $N\in\N$, 
\eq
\label{D3}
 \inf_{z\in\s_+}\Big(2\bs\V_N(z) + \frac{1}{2}\log(1-|z|^2)\Big)
=2 \inf_{x\in \R_+}\Big(V_N(x)-\log(1+x^2)\Big)>-\infty,
\qe
where the latter inequality follows from the definition \eqref{VN} of $V_N$ and the asymptotic behavior \eqref{asymptoBessel} of the Bessel function. Lemma \ref{boundedM} then follows from \eqref{D1}--\eqref{D3}.
 \end{proof}

 \begin{proof}[Proof of Lemma \ref{unifint}]  Since  $T_*$ is an homeomorphism from $\E(\R_-)$ to $\E(\s_-)$, we obtain with the metric relation \eqref{metric} that for any $\nu\in\E(\s_-)$
 \begin{align*}
 \int_{\s_-}\log|\xi|d\nu(\xi) & =\int_{\R_-}\log|T(u)|d{T_*}^{-1}\nu(u)\\
  & = \int_{\R_-}\log\left(\frac{|u|}{\sqrt{1+|u|^2}}\right)d{T_*}^{-1}\nu(u)\\
  & = \int _{|u|\leq 1}\log|u|d{T_*}^{-1}\nu(u)+ F(\nu),
 \end{align*}
 where $F$ is a continuous function on $\E(\s_-)$.  Lemma \ref{unifint} is thus equivalent to the continuity on $\E(\R_-)$ of the functional 
 \eq
 \label{logR}
 \nu\mapsto \int _{|u|\leq 1}\log|u|d\nu(u),
 \qe
which is itself equivalent to the uniformly integrability  of $u\mapsto \bs1_{|u|\leq 1}\log|u|$ with respect to the measures of $\E(\R_-)$, namely to
 \eq
 \label{D5'}
 \lim_{\varepsilon\rightarrow 0}\sup_{\nu\in\E(\R_-)}\int_{|u|\leq\,\varepsilon} \big|\bs1_{|u|\leq 1}\log|u|\,\big|d\nu(u)=0.
 \qe
 Since for any $\varepsilon>0$ and any $\nu\in\E(\R_-)$
 \[
\int_{|u|\leq\,\varepsilon} \big|\bs1_{|u|\leq 1}\log|u|\,\big|d\nu(u)\leq\frac{1}{|\log(\varepsilon)|}\int_{|u|\leq 1}\log^2|u|d\nu(u),
 \]
it is enough to show that 
 \eq
 \label{D5}
 \sup_{\nu\in\E(\R_-)}\int_{|u|\leq 1}\log|u|^2d\nu(u)<+\infty
 \qe
 in order to obtain \eqref{D5'}. By definition \eqref{defER} of $\E(\R_-)$ we have
 \begin{align}
 \label{D6}
 &  \sup_{\nu\in\E(\R_-)}\int_{|u|\leq 1}\log^2|u|d\nu(u)  \\
 \leq  &\qquad \max \left\{ \sup_{N}\int_{|u|\leq 1}\log^2|u|d\sigma_N(u)\;,\;\int_{|u|\leq 1}\log^2|u|d\sigma(u)\right\}\nonumber.
 \end{align}
 First, it follows from the definition \eqref{sigma} of $\sigma$ that 
 \eq
 \label{D7}
 \int_{|u|\leq 1}\log^2|u|d\sigma(u) = \frac{\sqrt{a}}{\pi}\int_0^1x^{1/2}\log^2(x)dx<+\infty.
 \qe
 Then, the definition \eqref{sigmaN} of $\sigma_N$ gives
 \eq
 \label{D8}
 \int_{|u|\leq 1}\log^2|u|d\sigma_N(u)= \frac{1}{N}\sum_{k\geq 0 \,:\; \frac{j_{\alpha,k}}{2\sqrt{a}N}\leq 1}\log^2\left(\frac{j_{\alpha,k}}{2\sqrt{a}N}\right)^2.
 \qe
It is a consequence of the  McMahon expansion formula \cite[formula 9.5.12]{AS} that
 \eq
 \label{McMahon2}
 \lim_{k\rightarrow\infty}\big(j_{\alpha,k+1}-j_{\alpha,k}\big)=\pi,
 \qe
and this provides the existence of $C>0$ independent of $N$  satisfying 
\begin{align}
\label{D9}
 & \frac{1}{N}\sum_{k \geq 0\,:\; \frac{j_{\alpha,k}}{2\sqrt{a}N}\leq 1}\log^2\left(\frac{j_{\alpha,k}}{2\sqrt{a}N}\right)^2  \nonumber \\
\leq & \quad  C\frac{j_{\alpha,0}}{2\sqrt{a}N}\log^2\left(\frac{j_{\alpha,0}}{2\sqrt{a}N}\right)^2 \nonumber \\
& \quad \quad+\; C\Big(\frac{j_{\alpha,k}}{2\sqrt{a}N} - \frac{j_{\alpha,k-1}}{2\sqrt{a}N}\Big)\sum_{k>0 \,:\; \frac{j_{\alpha,k}}{2\sqrt{a}N}\leq 1}\log^2\left(\frac{j_{\alpha,k}}{2\sqrt{a}N}\right)^2 \nonumber\\
\leq & \quad C\int_0^1\log^2(x^2)dx \;<+\infty.
\end{align}
Indeed, the latter inequality follows by splitting the integration domain and from the fact that $x\mapsto\log^2(x^2)$ is non-negative and decreasing on $[0,1]$.  Combining \eqref{D6}--\eqref{D8} and \eqref{D9} we obtain \eqref{D5}, which completes the proof of Lemma \ref{unifint}.

\end{proof}

\begin{proof}[Proof of Lemma \ref{CNlim}]
From the metric relations \eqref{metric2} we obtain
\begin{align}
\label{CNineq}
C_N =  & \; e^{3N/2}\int_{\s_+^N\times \s_-^{N/2}}\prod_{i=1}^N(1-|z_i|^2\,)d\lambda(z_i)\prod_{i=1}^{N/2}|\xi_i|\sqrt{1-|\xi_i|^2}\,d\eta_N(\xi_i)\nonumber\\
= & \; e^{3N/2}\left(\int_{\s_+} (1-|z|^2\,)d\lambda(z)\right)^N\left(\int_{\s_-}|\xi|\sqrt{1-|\xi|^2}\,d\eta_N(\xi)\right)^{N/2}\nonumber\\
= & \; e^{3N/2}\left(\int_{\R_+} \frac{1}{1+x^2}dx\right)^N\left(\int_{\R_-}\frac{|u|}{1+u^2}d\sigma_N(u)\right)^{N/2}.
\end{align}
Since the definition \eqref{sigmaN} of $\sigma_N$ yields
\[
\int_{\R_-} \frac{|u|}{1+u^2}d\sigma_N(u) \leq \int_{\R_-}\frac{1}{|u|}d\sigma_N(u)= 4aN\sum_{k=0}^\infty\frac{1}{j_{\alpha,k}^2},
\]
then Lemma \ref{CNlim} follows from \eqref{CNineq} and  the identity \cite[Section 15.51]{W}
\[
\sum_{k=0}^\infty\frac{1}{j_{\alpha,k}^2} = \frac{1}{4(1+\alpha)}<+\infty.
\]
\end{proof}

\begin{proof}[Proof of Lemma \ref{liminfineq}] 
We write 

\eq
\label{ineqA}
\liminf_{N\rightarrow\infty}\inf_{\B_\delta(\mu,\nu)} J_N^M \geq \inf_{\B_\delta(\mu,\nu)} J^M +\liminf_{N\rightarrow\infty}\inf_{\B_\delta(\mu,\nu)}\big(J^M_N-J^M\big)
\qe
and note that, from the definitions \eqref{JMN} and \eqref{JM} of $J^M_N$ and $J^M$ respectively, we have
\begin{align}
\label{ineqB}
 & \inf_{\B_\delta(\mu,\nu)}\big(J^M_N-J^M\big) \geq \frac{1}{2}\inf_{(z,\xi)\in\s_+\times\s_-}\Big\{\min \Big(2 \bs\V_N(z) +\log|z-\xi|-\log|\xi|, M\Big) \nonumber\\
& \qquad\qquad\qquad  \qquad\qquad\qquad \qquad -\min \Big(2 \bs\V(z) +\log|z-\xi|-\log|\xi|, M\Big)\Big\}. 
\end{align}

The inequality  \eqref{D1}  and the fast growth of $\bs\V(z)$ and $\bs\V_N(z)$ as $z\rightarrow(0,1)$, which follows from the definitions \eqref{curlyV1}--\eqref{curlyV2}, \eqref{curlyVN1}--\eqref{curlyVN2} and the asymptotic behavior \eqref{asymptoBessel}, provide  the existence of a neighborhood $\mathcal{N}_\infty\subset\s_+$ of $(0,1)$ such that  for all $N$ 
\begin{align}
\label{ineqC}
 &\min \Big(2 \bs\V_N(z) +\log|z-\xi|-\log|\xi|, M\Big) \\
&\;=\; \min \Big(2 \bs\V(z) +\log|z-\xi|-\log|\xi|, M\Big) = M, \qquad (z,\xi)\in\mathcal{N}_\infty\times\s_-.\nonumber
\end{align}
Next, we claim the existence of a subset $\mathcal{N}_0\subset\s_+$  satisfying $\mathcal N_0\cup\mathcal N_\infty=\s_+$ and    
\begin{align}
\label{ineqD}
 &\min \Big(2 \bs\V_N(z) +\log|z-\xi|-\log|\xi|, M\Big) \\
&\;\geq\; \min \Big(2 \bs\V(z) +\log|z-\xi|-\log|\xi|, M\Big), \qquad (z,\xi)\in\mathcal{N}_0\times\s_-,\nonumber
\end{align}
for any $N$ sufficiently large, so that Lemma \ref{liminfineq} would follow by combining \eqref{ineqA}--\eqref{ineqD}.

To show this it is enough to prove that for any $L>0$  there exists  $N_L\geq 0$ such that for all $N\geq N_L$
\[
V_N(x)-x+2\sqrt{ax}\;\geq 0, \qquad x\in[0,L],
\]
or equivalently (see the definitions \eqref{VN} and \eqref{weight})
\eq
\label{ineqE}
y^{\alpha}I_\alpha(y)e^{-y}\leq (2N\sqrt{a}\,)^\alpha, \qquad  y\in[0,2N\sqrt{aL}].
\qe
Indeed, if we choose $\mathcal N_0=T([0,L])$ with  $L$ large enough so that $\mathcal N_0\cup\mathcal N_\infty=\s_+$,  then \eqref{ineqD} would hold for any $N\geq N_L$ as a consequence of \eqref{ineqE}. Given $L>0$, if $\alpha=0$ then \eqref{ineqE} holds because $I_0(0)=1$ and $y\mapsto I_0(y)e^{-y}$ is decreasing on $\R_+$. If $\alpha>0$, it is then easy to see from the asymptotic behavior $y^{\alpha}I_\alpha(y)e^{-y}=(2\pi)^{-1/2}y^{\alpha-1/2}(1+O(y^{-1}))$ as $y\rightarrow+\infty$, provided by \eqref{asymptoBessel}, that  \eqref{ineqE} is satisfied for any $N$ large enough.  This completes the proof of Lemma \ref{liminfineq}.

\end{proof}

We now provide a proof for the announced LDP lower bound.

\subsection{A LDP lower bound for $(\mu^N,\nu^N)_N$}
\label{LDPLB}

The aim of this section is to establish the following.

\begin{proposition}
\label{LDPLBp}
For any open set $\mathcal{O} \subset\Ma\times\Me$
\[
\liminf_{N\rightarrow\infty}\frac{1}{N^{2}}\log \Big\{Z_N\p_N\Big(\left(\mu^{N},\nu^{N}\right)\in \mathcal{O}\Big)\Big\} \geq-\inf_{(\mu,\,\nu)\in\mathcal{O}}\J(\mu,\nu).
\]
\end{proposition}

\begin{proof} 
Note that it is sufficient to show that for all  $(\mu,\nu)\in\mathcal{O}$
\eq
\label{L1}
\liminf_{N\rightarrow\infty}\frac{1}{N^{2}}\log \Big\{Z_N\p_N\Big(\left(\mu^{N},\nu^{N}\right)\in \mathcal{O}\Big)\Big\} \geq-\J(\mu,\nu).
\qe
We first prove in two steps that  \eqref{L1} holds if $\mu$ and $\nu$ satisfy the following :
\begin{assumption}\
\label{assump}
\begin{enumerate}
\item[{\rm (1)}]
 $\mu$ and $\nu$ have compact support.
 \item[{\rm (2)}]
 $\Supp(\mu)\subset\R_+\setminus\{0\}$ and $\Supp(\nu)\subset\R_-\setminus\{0\}$.
 \item[{\rm (3)}]
 With $\sigma$ as in \eqref{sigma}, there exists $0<\varepsilon<1$ such that $\nu\leq (1-\varepsilon)\sigma$.
 \item[{\rm (4)}]
 $T_*\mu$ and $T_*\nu$ have finite logarithmic energy.
\end{enumerate}
\end{assumption}
\noindent
We then extend in a last step \eqref{L1} to all $(\mu,\nu)\in\mathcal O$ by mean of an approximation procedure. This approach is similar to the strategy developed in \cite[Section 3.2]{BAG}, see also  \cite[Section 3.4]{GM}.

\paragraph{Step 1 (Discretization)} Given $(\mu,\nu)\in\mathcal O$ satisfying Assumption \ref{assump}, our first step consists to build discrete approximations of $(\mu,\nu)$. To this aim, we note that $\mu$ and $\nu$ have no atom as a consequence of Assumption \ref{assump} (d) and consider 
\begin {align}
\label{discrx1}
x_1^{(N)}  & = \min \Big\{x \in\R_+: \; \mu\big([0,x]\big)= \frac{1}{N}\Big\}, \\
\label{discrx2}
x^{(N)}_{i+1}  & = \min \Big\{x \geq x_i^{(N)} : \; \mu\big([x_{i}^{(N)},x]\big)= \frac{1}{N}\Big\}, \qquad i=1,\ldots,N-1,
\end{align}
and similarly 
\begin {align}
\label{discry1}
y_1^{(N)}  & = \min \Big\{y \in\R_-: \; \nu\big((-\infty,y]\big)= \frac{1}{N}\Big\}, \\
\label{discry2}
y_{i+1}^{(N)}  & = \min \Big\{ y\geq y_i^{(N)} : \; \nu\big([y_i^{(N)},y]\big)= \frac{1}{N}\Big\}, \qquad i=1,\ldots,N/2-1.
\end{align}
Since $\mu$ and $\nu$ moreover have compact supports,  the following weak convergence   follows easily 
\eq
\label{convdiscr}
\lim _{N\rightarrow\infty}\frac{1}{N}\sum_{i=1}^N\delta(x_i^{(N)})=\mu \qquad \mbox{and} \qquad \lim _{N\rightarrow\infty}\frac{1}{N}\sum_{i=1}^{N/2}\delta(y_i^{(N)}) =\nu.
\qe
Because the $u_i$'s are distributed on the discrete set $\A_N$ \eqref{AN}, we also set 
\eq
\label{defui}
u_i^{(N)} = \max\Big\{ u\in\A_N : \; u < y_i^{(N)}\Big\},\quad i=1,\ldots,N/2,
\qe
and moreover introduce 
\eq
\nu^{(N)}=\frac{1}{N}\sum_{i=1}^{N/2}\delta(u_i^{(N)}).
\qe
We now show that, for any $N$ large enough, the $u_i^{(N)}$'s lie in the convex hull $co(\Supp(\nu))$ and the following interlacing property holds
\eq
\label{interlace}
y_i^{(N)}< u_{i+1}^{(N)} < y_{i+1}^{(N)}, \qquad i=1,\ldots,N/2-1.
\qe
Indeed, with $\varepsilon$ as in Assumption \ref{assump} (3), \eqref{McMahon2} yields  $k_\varepsilon$ such that 
\[
\sup_{k\geq k_\varepsilon}\big(j_{\alpha,k+1}-j_{\alpha,k}\big)\leq\pi(1+\varepsilon)
\]
and, since $0\notin \Supp(\nu)$ by assumption, there exists $N_\varepsilon$ such that 
\[
\sup_{k< k_\varepsilon}\nu\big([a_{k+1,N},a_{k,N}]\big) = 0, \qquad N\geq N_\varepsilon.
\]
Thus, recalling the definition \eqref{akN} of the $a_{k,N}$'s, we obtain for any $N\geq N_\varepsilon$
\begin{align*}
\sup_{k\geq 0}\nu\big([a_{k+1,N},a_{k,N}]\big) & = \sup_{k\geq k_\varepsilon}\nu\big([a_{k+1,N},a_{k,N}]\big) \\
 & \leq  (1-\varepsilon)\sup_{k\geq k_\varepsilon}\sigma\big([a_{k+1,N},a_{k,N}]\big)\\
 & =  (1-\varepsilon)\frac{1}{\pi N}\sup_{k\geq k_\varepsilon}(j_{\alpha,k+1}-j_{\alpha,k}\big)\\
 & \leq (1-\varepsilon^2)\frac{1}{N}.
\end{align*}
The latter inequality implies that there exists an element of $\A_N$ in each $(y_i^{(N)}, y_{i+1}^{(N)})$ provided $N$ is large enough, so that \eqref{interlace} follows from the definition \eqref{defui} of the $u_i$'s, and moreover that all the $u_i$'s are in $co(\Supp(\nu))$.

Note that \eqref{interlace} yields $\nu^{(N)}\leq \sigma_N$, and thus $\nu^{(N)}\in\E(\R_-)$ for all $N$. Moreover, by combining   \eqref{interlace} with \eqref{convdiscr}, we obtain the weak convergence of $(\nu^{(N)})_N$  towards $\nu$.
As the result of the discretization step, we have shown the existence of $\delta_0>0$ and $N_0$ such that for all $0<\delta\leq\delta_0$ and $N\geq N_0$
\eq
\label{minoration}
\left\{ \Big(\,\frac{1}{N}\sum_{i=1}^{N}\delta(x_{i}), \nu^{(N)}\Big)\,:\; \bs x\in\R_+^N,\;\max_{i=1}^N |x_i-x_i^{(N)}| \leq \delta \right\} \subset \mathcal{O}.
\qe
 
 \paragraph{Step 2. (Lower bound)}
We now prove \eqref{L1} when $(\mu,\nu)$ satisfies Assumption \ref{assump}. As a consequence of  \eqref{minoration} we obtain for any $0<\delta\leq\delta_0$
\begin{multline}
\label{L5}
Z_N \p_N\Big(\left(\mu^{N},\nu^{N}\right)\in \mathcal{O}\Big)\\
\geq \int_{\big\{\bs x \in\R_+^N\, :\, \max_i|x_i-x_i^{(N)}|\leq \delta\big\}}\frac{\Delta_{N}^{2}\left(\bs x\right)\Delta_{N/2}^{2}\big(\bs u^{(N)}\big)}{\Delta_{N,N/2}\big(\bs x,\bs u^{(N)}\big)}\prod_{i=1}^{N/2}|u_i^{(N)}| \prod_{i=1}^N e^{-NV_N(x_i)}dx_i.
\end{multline}
For a Borel measure $\lambda$ on $\R$ with compact support, introduce its logarithmic potential  
\[
U^{\lambda}(x)=\int\log\frac{1}{|x-u|}d\lambda(u)
\] 
which is continuous on $\R\setminus\Supp(\lambda)$ \cite[Chapter 0]{ST} and note that
\eq
\label{L5'}
\Delta_{N,N/{2}}\big(\bs x,\bs u^{(N)}\big) = \prod_{i=1}^N\exp\Big\{-N\, U^{\nu^{(N)}}(x_i)\Big\}.
\qe
We also set for $x\in \R_+$
\begin{align}
W_N(x) & = V_N(x) - U^{\nu^{(N)}}(x), \\
\label{L5''}
W(x) \;& = x -2\sqrt{ax} - U^{\nu}(x)
\end{align}
and obtain from \eqref{L5}--\eqref{L5''} 
\begin{align}
\label{L6}
& Z_N \p_N\Big(\left(\mu^{N},\nu^{N}\right)\in \mathcal{O}\Big) \nonumber \\
\geq  & \quad \exp\Big\{-N^2\max_{x\in co(\Supp(\mu))}|W_N(x)-W(x)| \Big\} \Delta_{N/2}^{2}\big(\bs u^{(N)}\big)|a_{0,N}|^{N/2} \\
 & \qquad\qquad\times\int_{\big\{\bs x \in\R_+^N\, :\, \max_i|x_i-x_i^{(N)}|\leq \delta\big\}}\Delta_{N}^{2}\left(\bs x\right) \prod_{i=1}^N e^{-NW(x_i)}dx_i.\nonumber
\end{align}
By using the change  of variables $x_i\mapsto x_i+x_i^{(N)}$ for $i=1,\ldots,N$, and the fact that $|x_i^{(N)}-x^{(N)}_j+x_i-x_j|\geq \max\big\{|x^{(N)}_i-x_j^{(N)}|\, ,\, |x_i-x_j|\big\}$ as soon as $x_i\geq x_j$ and $x_i^{(N)}\geq x_j^{(N)}$, we find
\begin{align}
\label{L7}
  & \int_{\big\{\bs x \in\R_+^N\, :\, \max_i|x_i-x_i^{(N)}|\leq \delta\big\}}\Delta_{N}^{2}\left(\bs x\right)\prod_{i=1}^N e^{-NW(x_i)}dx_i
  \nonumber\\
\geq & \quad \int_{ [0,\delta]^N}\Delta_{N}^{2}\big(\bs x +\bs x^{(N)}\big) \prod_{i=1}^N e^{-NW(x_i+x_i^{(N)})}dx_i\nonumber\\
\geq  &  \quad\prod_{i+1<j}\big(x_j^{(N)}-x_i^{(N)}\big)^2\prod_{i=1}^{N-1}\big(x_{i+1}^{(N)}-x_{i}^{(N)}\big)\prod_{i=1}^N e^{-NW(x_i^{(N)})}\\
 & \qquad \times \int_{ \big\{ \bs x \in [0,\delta]^N \, : \; x_1< \cdots \,<x_{N}\big\}}\prod_{i=1}^{N-1}(x_{i+1}-x_i)\prod_{i=1}^N e^{-N|W(x_i+x_i^{(N)})-W(x_i^{(N)})|}dx_i.\nonumber
\end{align}
Since  the $x_i^{(N)}$'s lie in the compact set $co(\Supp(\mu))$ and  $W$ is  continuous there, we  obtain 
\eq
\label{L8}
\lim_{\delta\rightarrow0}\limsup_{N\rightarrow\infty}\max_{1\leq i \leq N}\max_{x\in[0,\delta]}|W(x+x_i^{(N)})-W(x_i^{(N)})|=0
\qe
and also, using moreover \eqref{convdiscr},  
\eq
\label{L9}
\lim_{N\rightarrow\infty}\frac{1}{N}\sum_{i=1}^N W\big(x_i^{(N)}\big) = \int W(x)d\mu(x).
\qe
Using the change of variables $u_1=x_1$ and $u_{i+1}=x_{i+1}-x_i$ for $i=1,\ldots,N-1$, it follows 
\begin{align}
\label{L10}
  & \int_{ \big\{ \bs x \in [0,\delta]^N \, : \; x_1< \cdots \,<x_{N}\big\}}dx_1\prod_{i=1}^{N-1}(x_{i+1}-x_i)dx_{i+1} \nonumber \\
\geq &\;  \int_{[0,\,\delta/N]^N}du_1\prod_{i=2}^{N}u_{i}du_i \; = \; \frac{1}{2^{N-1}}\left(\frac{\delta}{N}\right)^{2N-1}.
\end{align}
We thus obtain from \eqref{L7}--\eqref{L10}
\begin{align}
\label{control1}
&\liminf_{\delta\rightarrow 0} \liminf_{N\rightarrow\infty}\frac{1}{N^2}\log\int_{\big\{\bs x \in\R_+^N\, :\, \max_i|x_i-x_i^{(N)}|\leq \delta\big\}}\Delta_{N}^{2}\left(\bs x\right)\prod_{i=1}^N e^{-NW(x_i)}dx_i \nonumber\\
\geq & \;\liminf_{N\rightarrow\infty} \frac{1}{N^2}\left( \sum_{i+1<j}\log \big(x_j^{(N)}-x_i^{(N)}\big)^2+ \sum_{i=1}^{N-1}\log \big(x_{i+1}^{(N)}-x_i^{(N)}\big)\right)\\
&  \qquad \qquad  \qquad \qquad \qquad\qquad\qquad \qquad \qquad - \int W(x)d\mu(x) \nonumber.
\end{align}
Next, we have
\eq
\label{control2}
\lim_{N\rightarrow\infty}\;\max_{x\in co(\Supp(\mu))}|W_N(x)-W(x)| = 0.
\qe
Indeed, the  asymptotic behavior \eqref{asymptoBessel} yields the uniform convergence of $V_N(x)$ towards $x-2\sqrt{ax}$  as $N\rightarrow\infty$ on every compact subset of $\R_+\setminus\{0\}$, and in particular on $co(\Supp(\mu))$. It is thus enough to show the uniform convergence of $U^{\nu^{(N)}}$ to $U^\nu$ on $co(\Supp(\mu))$ to obtain \eqref{control2}. For any $x\in \Supp(\mu)$, the map $y\mapsto \log|x-y|$ is continuous and bounded on $co(\Supp(\nu))$, so that the pointwise convergence of $U^{\nu^{(N)}}$ to $U^{\nu}$ on $co(\Supp(\mu))$ follows from from the weak convergence of $\nu^{(N)}$ to $\nu$. Since for all $N$ the map $U^{\nu^{(N)}}$ is continuous and decreasing on the compact $co(\Supp(\mu))$, and that $U^{\nu}$ is moreover continuous there, the pointwise convergence extends to the uniform convergence by  Dini's theorem.

We thus obtain from \eqref{control1}--\eqref{control2} by taking the limit $N\rightarrow\infty$ and then $\delta\rightarrow 0$ in \eqref{L6}   that 
\begin{align}
\label{L16}
 & \liminf_{N\rightarrow\infty}\frac{1}{N^{2}}\log \Big\{Z_N\p_N\Big(\left(\mu^{N},\nu^{N}\right)\in \mathcal{O}\Big)\Big\}\nonumber\\
\geq & \;\liminf_{N\rightarrow\infty} \frac{1}{N^2}\left( \sum_{i+1<j}\log \big(x_j^{(N)}-x_i^{(N)}\big)^2+ \sum_{i=1}^{N-1}\log \big(x_{i+1}^{(N)}-x_i^{(N)}\big)\right)\\
& + \quad \liminf_{N\rightarrow\infty} \frac{1}{N^2} \sum_{i<j}\log\big(u_j^{(N)}-u_i^{(N)}\big)^2 - \int W(x)d\mu(x) \nonumber.
\end{align}
Now, note that because $x\mapsto \log(x)$ increases on $\R_+$ the definition \eqref{discrx1}--\eqref{discrx2} of the $x_i^{(N)}$'s yields
\begin{align}
\label{L14}
 & \frac{1}{N^2}\sum_{i+1 < j} \log\big(x^{(N)}_{j}-x_i^{(N)}\big)^2 +\frac{1}{N^2}\sum_{i=1}^{N-1} \log\big(x^{(N)}_{i+1}-x_i^{(N)}\big)
 \nonumber\\
= & \quad 2\sum_{1\leq i \leq j \leq N-1} \log\big(x^{(N)}_{j+1}-x_i^{(N)}\big)\iint_{[x_i^{(N)},x_{i+1}^{(N)}]\times[x_j^{(N)},x_{j+1}^{(N)}]}\boldsymbol 1_{x<y}\,d\mu(x)d\mu(y)\nonumber \\
\geq & \quad 2 \iint_{x_1^{(N)}\leq \,x\,<\, y\, \leq \,x_N^{(N)}}\log(y-x)d\mu(x)d\mu(y)
\end{align}
and then that 
\begin{align}
\label{L15}
&2 \lim_{N\rightarrow\infty}\iint_{x_1^{(N)}\leq \,x\,<\, y\, \leq \,x_N^{(N)}}\log(y-x)d\mu(x)d\mu(y) \nonumber\\
&=\qquad \iint\log|x-y|d\mu(x)d\mu(y).
\end{align}
The interlacing property \eqref{interlace} yields
\[
u_j^{(N)}-u_i^{(N)} \geq \;y_{j-1}^{(N)}-y_i^{(N)} \quad\qquad\mbox{for  } i+1<j,
\]
and thus
\eq
\label{L11}
\sum_{i<j}\log\big(u_j^{(N)}-u_i^{(N)}\big)^2 \geq \sum_{i=1}^{N/2-1}\log\big(u_{i+1}^{(N)}-u_i^{(N)}\big)^2+\sum_{i+1<j}\log\big(y_{j-1}^{(N)}-y_i^{(N)}\big)^2.
\qe
Since 
\begin{align*}
\min_{1\leq i \leq N/2}\big(u_{i+1}^{(N)}-u_i^{(N)}\big) & \geq \;\inf_{k\geq 0}\big(a_{k,N} - a_{k+1,N}\big) \\
& \geq \;\frac{j_{\alpha,0}}{2aN^2}\inf_{k\geq 0}\big(j_{\alpha,k+1}-j_{\alpha,k}\big),
\end{align*}
we obtain from \eqref{McMahon2} and \eqref{L11}  
\eq
\label{L12}
\liminf_{N\rightarrow\infty}\frac{1}{N^2}\sum_{i<j}\log\big(u_j^{(N)}-u_i^{(N)}\big)^2 \geq \liminf_{N\rightarrow\infty}\frac{1}{N^2}\sum_{i+1<j}\log\big(y_{j-1}^{(N)}-y_i^{(N)}\big)^2.
\qe 
Moreover,  because for $1 \leq i \leq N/2-1$
\begin{align*}
y_{i+1}^{(N)}-y_i^{(N)} & \geq \;2|\max(\Supp(\nu))|^{1/2}\big(|y_i^{(N)}|^{1/2} - |y_{i+1}^{(N)}|^{1/2}\, \big)\\
& = \;\frac{\pi}{\sqrt{a}}|\max(\Supp(\nu))|^{1/2}\,\sigma\big([y_i^{(N)},y_{i+1}^{(N)}]\big)\\
& \geq \; \frac{\pi}{\sqrt{a}}|\max(\Supp(\nu))|^{1/2}\,\nu\big([y_i^{(N)},y_{i+1}^{(N)}]\big)\\
& = \;\frac{\pi}{\sqrt{a}N}|\max(\Supp(\nu))|^{1/2},
\end{align*}
we obtain from \eqref{L12}
\eq
\label{L12bis}
\liminf_{N\rightarrow\infty}\frac{1}{N^2}\sum_{i<j}\log\big(u_j^{(N)}-u_i^{(N)}\big)^2 \geq \liminf_{N\rightarrow\infty}\frac{1}{N^2}\sum_{i+2<j}\log\big(y_{j-1}^{(N)}-y_i^{(N)}\big)^2.
\qe 
Next, similarly than in \eqref{L14}--\eqref{L15}, we obtain from the definition \eqref{discry1}--\eqref{discry2} of the $y_i^{(N)}$'s  that
\begin{align}
\label{L13}
&\liminf_{N\rightarrow\infty}\frac{1}{N^2}\sum_{i+2<j}\log\big(y_{j-1}^{(N)}-y_i^{(N)}\big)^2 \nonumber\\
= & \quad 2\liminf_{N\rightarrow\infty}\sum_{i+2<j}\log\big(y_{j-1}^{(N)}-y_i^{(N)}\big)\iint _{[y^{N}_i,y_{i+1}^{(N)}]\times[y^{(N)}_{j-2},y_{j-1}^{(N)}]}\boldsymbol 1_{u<v}\,d\nu(u)d\nu(v) \nonumber\\
\geq & \quad 2\liminf_{N\rightarrow\infty} \iint _{y_1^{(N)}\leq u < v \leq y_{N/2-1}^{(N)}}\log(v-u)d\nu(u)d\nu(v)\nonumber\\
= & \quad\iint\log|x-y|d\nu(x)d\nu(y).
\end{align}
From \eqref{L16}--\eqref{L15} and \eqref{L12bis}--\eqref{L13} it follows

\begin{align}
\label{L17}
 & \liminf_{N\rightarrow\infty}\frac{1}{N^{2}}\log \Big\{Z_N\p_N\Big(\left(\mu^{N},\nu^{N}\right)\in \mathcal{O}\Big)\Big\}\nonumber\\
\geq & \qquad- \Bigg\{\iint\log\frac{1}{|x-y|}d\mu(x)d\mu(y)+\int W(x)d\mu(x)\\
& \qquad \qquad\qquad\qquad \qquad  \qquad +\iint\log\frac{1}{|x-y|}d\nu(x)d\nu(y) \Bigg\} \nonumber.
\end{align}
Since both $V$ and $U^\nu$ are  bounded and continuous functions on the compact $\Supp(\mu)$, by \eqref{L5''}

\begin{align*}
\int W(x)d\mu(x) & = \int \Big(x-2\sqrt{ax}\,\Big)d\mu(x) - \int U^{\nu}(x)d\mu(x)\nonumber\\
& = \int \Big(x-2\sqrt{ax}\,\Big)d\mu(x) - \iint\log\frac{1}{|x-y|}d\mu(x)d\nu(y),
\end{align*}
and thus

\begin{align}
\label{L18}
 & \liminf_{N\rightarrow\infty}\frac{1}{N^{2}}\log \Big\{Z_N\p_N\Big(\left(\mu^{N},\nu^{N}\right)\in \mathcal{O}\Big)\Big\}\nonumber\\
\geq & \qquad- \Bigg\{\iint\log\frac{1}{|x-y|}d\mu(x)d\mu(y) - \iint\log\frac{1}{|x-y|}d\mu(x)d\nu(y)\\
& \qquad \qquad \qquad  \qquad +\iint\log\frac{1}{|x-y|}d\nu(x)d\nu(y)+\int \Big(x-2\sqrt{ax}\,\Big)d\mu(x) \Bigg\} \nonumber.
\end{align}
Since by assumption the measures $\mu$, $\nu$ have  compact supports and $T_*\mu$, $T_*\nu$ have finite logarithmic energies, then $\mu$, $\nu$  also have finite logarithmic energies and clearly
\[
\int\log(1+x^2)d\mu(x)<+\infty,\qquad \int\log(1+x^2)d\nu(x)<+\infty.
\]
Thus, one can use the relation \eqref{relationCsphere} and obtain that the right-hand side of \eqref{L18} equals $\J(\mu,\nu)$, see \eqref{ratefunction}, which proves \eqref{L1}.

\paragraph{Step 3. (Approximation)}
First note that \eqref{L1} trivially  holds  as soon as  $\J(\mu,\nu)=+\infty$. It is thus enough to show \eqref{L1}  when both $T_*\mu$ and $T_*\nu$ have finite logarithmic energy, and one can moreover assume that $\nu\leq \sigma$. For such $(\mu,\nu)$, we now construct a sequence $(\mu_k,\nu_k)_k$ of $\M_1(\R_+)\times\E(\R_-)$ where each $(\mu_k,\nu_k)$ satisfies Assumption \ref{assump}, such that  we have the weak convergences
\[
\lim_{k\rightarrow\infty}\mu_k=\mu,\qquad \lim_{k\rightarrow\infty}\nu_k=\nu,
\] 
and which moreover satisfies
\eq
\label{goodapprox}
\lim_{k\rightarrow\infty}\J(\mu_k,\nu_k)=\J(\mu,\nu).
\qe
This, combined with the two first steps of the proof, shows that \eqref{L1} actually holds for all $(\mu,\nu)\in\mathcal O$, and thus complete the proof of Proposition \ref{LDPLBp}.

For any $k$ large enough, let $\mu_k\in\M_1(\R_+)$ be the normalized restriction of $\mu$ to $[k^{-1},k]$, so that $\Supp(\mu_k)\subset \R_+\setminus\{0\}$ is compact. The monotone convergence theorem yields  that $(\mu_k)_k$ converges to $\mu$ as $k\rightarrow\infty$. To approximate $\nu$, we have to stay in the class of constrained measures $\M_{1/2}^{\sigma}(\R_-)$, and thus to proceed a bit more carefully. To this aim, choose two sequences  $(a_k)_k$ and $(b_k)_k$  satisfying  $a_k<b_k<0$ and  
\begin{itemize}
\item[1)]
$a_k$ decreases to $\inf(\Supp(\nu)) $ as $k\rightarrow\infty$,
\item[2)]
 $b_k$ increases to $\max(\Supp(\nu))$ as $k\rightarrow\infty$,
\item[3)] for any $k$ large enough,
\eq
\label{condakbk}
\nu\big([a_k,b_k]\big)\geq(1-k^{-1}).
\qe
\end{itemize}
Since $\nu\leq \sigma$, the Radon-Nikodym theorem yields $f \in L^1(\R_-)$ such that 
\eq
\label{constrdensity}
d\nu(x)=f(x)dx,\qquad f(x)\leq \frac{\sqrt{a}}{\pi}|x|^{-1/2}, \qquad x\in\R_-.
\qe
We then set the probability measure
\eq
\label{nukapprox}
d\nu_k(x)=\left(\frac{(1-k^{-1})^4}{\nu([a_k,b_k])}\right)f\big((1-k^{-1})^4x\big) \bs{1}_{[a_k,b_k]}\big((1-k^{-1})^4x\big)dx,
\qe
whose support $\Supp(\nu_k)\subset \R_-\setminus \{0\}$ is compact. $(\nu_k)_k$ is easily seen to converge to $\nu$ as $k\rightarrow\infty$ using monotone convergence. Moreover, it follows from \eqref{condakbk}--\eqref{constrdensity} and the definition \eqref{nukapprox} that 
\[
\nu_k\leq (1-k^{-1})\sigma.
\] 
The fact that  $T_*\mu_k$ and  $T_*\nu_k$ have finite logarithmic energy for $k$ large enough, and thus that $(\mu_k,\nu_k)$ satisfies Assumption \ref{assump}, will be a consequence of \eqref{G1}--\eqref{G2}, see below.

We now prove that the sequence $(\mu_k,\nu_k)_k$ satisfies  \eqref{goodapprox}. Recall that   $\J(\mu,\nu)=J(T_*\mu,T_*\nu)$ where $J$ is as in \eqref{representation}, namely
\begin{align}
\label{G0}
\J(\mu_k,\nu_k) = & \iint\log\frac{1}{|z-w|}dT_*\mu_k(z)dT_*\mu_k(w)\\
& \quad  +\iint \Big(2\bs\V(z)+\log|z-\xi|-\log|\xi|\Big)dT_*\mu_k(z)dT_*\nu_k(\xi) \nonumber\\
 & \qquad\qquad +\iint\log\frac{1}{|\xi-\zeta|}dT_*\nu_k(\xi)dT_*\nu_k(\zeta)+\int\log|\xi|dT_*\nu_k(\xi)\nonumber.
\end{align}
First, since $\s$ is compact, we obtain by monotone convergence
\begin{align}
\label{G1}
&\lim_{k\rightarrow\infty}\iint\log\frac{1}{|z-w|}dT_*\mu_k(z)dT_*\mu_k(w)\nonumber\\
= &\; \lim_{k\rightarrow\infty}\int_{k^{-1}}^k\int_{k^{-1}}^k\log\frac{1}{|T(x)-T(y)|}d\mu(x)d\mu(y)\nonumber\\
= &\; \iint\log\frac{1}{|T(x)-T(y)|}d\mu(x)d\mu(y)\nonumber\\
=& \; \iint\log\frac{1}{|z-w|}dT_*\mu(z)dT_*\mu(w).
\end{align}
Similarly, but using moreover the metric relation \eqref{metric}, the change of variables $u\mapsto u/(1-k^{-1})^4$ and the inequality $|u-v|\leq\sqrt{1+u^2}\sqrt{1+v^2}$,
\begin{align}
\label{G2}
&\lim_{k\rightarrow\infty}\iint\log\frac{1}{|\xi-\zeta|}dT_*\nu_k(\xi)dT_*\nu_k(\zeta)\nonumber \\
= &\; \lim_{k\rightarrow\infty}\iint\log\frac{\sqrt{1+u^2}\sqrt{1+v^2}}{|u-v|}d\nu_k(u)d\nu_k(v)\nonumber \\
= &\; \lim_{k\rightarrow\infty}\int_{a_k}^{b_k}\int_{a_k}^{b_k}\log\frac{\sqrt{(1-k^{-1})^8+u^2}\sqrt{(1-k^{-1})^8+v^2}}{|u-v|}d\nu(u)d\nu(v)\nonumber \\
= &\; \iint\log\frac{\sqrt{1+u^2}\sqrt{1+v^2}}{|u-v|}d\nu(u)d\nu(v) \nonumber \\
=& \; \iint\log\frac{1}{|\xi-\zeta|}dT_*\nu(\xi)dT_*\nu(\zeta).
\end{align}
The same arguments moreover combined  with the inequality $|x-u|\geq |u|$ for $(x,u)\in\R_+\times\R_-$  yield
\begin{align}
\label{G3}
& \lim_{k\rightarrow\infty}\iint \Big(2\bs\V(z)+\log|z-\xi|-\log|\xi|\Big)dT_*\mu_k(z)dT_*\nu_k(\xi)  \nonumber \\
= & \;\lim_{k\rightarrow\infty} \iint\Big\{2\big(x-2\sqrt{ax}-\log(1+x^2)\big)+\log|x-u|-\log|u|\Big\}d\mu_k(x)d\nu_k(u) \nonumber\\
= & \;\lim_{k\rightarrow\infty}  \int_{k^{-1}}^k\int_{a_k}^{b_k}\Big\{2\big(x-2\sqrt{ax}-\log(1+x^2)\big)\nonumber\\
& \qquad \qquad \qquad\qquad \qquad \qquad+\log|(1-k^{-1})^4x-u|-\log|u|\Big\}d\mu(x)d\nu(u) \nonumber\\
= &\; \iint\Big\{2\big(x-2\sqrt{ax}-\log(1+x^2)\big)+\log|x-u|-\log|u|\Big\}d\mu(x)d\nu(u) \nonumber\\
= &\; \iint \Big(2\bs\V(z)+\log|z-\xi|-\log|\xi|\Big)dT_*\mu(z)dT_*\nu(\xi)
\end{align}
After that, the continuity of $T_*$ on $\E(\s_-)$ and Lemma \ref{unifint} provide  
\eq
\label{G4}
\lim_{k\rightarrow\infty}\int \log|\xi|dT_*\nu_k(\xi) = \int \log|\xi|dT_*\nu(\xi).
\qe
Finally, \eqref{goodapprox} follows from  \eqref{G0}--\eqref{G4}, which completes  the  proof of Proposition \ref{LDPLBp}.
\end{proof}

\subsection{Proof of Theorem \ref{generalLDP} (c), (d)}
\label{fullLDP}

We are now in position the prove Theorem \ref{generalLDP} (c), (d). The following proof follows closely \cite[Section 2.3]{H}.
\begin{proof}[Proof of Theorem \ref{generalLDP} {\rm (c)}, {\rm (d)}]

It is enough to show that for any closed set $\F\subset\M_1(\R_+)\times\E(\R_-)$,
\eq
\label{UB}
\limsup_{N\rightarrow\infty}\frac{1}{N^2}\log \Big\{Z_N \p_N\Big((\mu^N,\nu^N)\in\F\Big)\Big\}\leq - \inf_{(\mu,\nu)\in\F}\J(\mu,\nu),
\qe
and for any open set $\mathcal{O}\subset \M_1(\R_+)\times\E(\R_-)$,
\eq
\label{LB}
\liminf_{N\rightarrow\infty}\frac{1}{N^2}\log\Big\{ Z_N \p_N\Big((\mu^N,\nu^N)\in\mathcal{O}\Big)\Big\}\geq - \inf_{(\mu,\nu)\in\mathcal{O}}\J(\mu,\nu).
\qe
Indeed, by taking $\F=\mathcal{O}=\M_1(\R_+)\times\E(\R_-)$ in \eqref{UB} and \eqref{LB}, one  obtains
\[
\lim_{N\rightarrow\infty}\frac{1}{N^2}\log Z_N= -\inf_{(\mu,\nu)\in\M_1(\R_+)\times\E(\R_-)}\J(\mu,\nu)=-\J(\mu^*,\nu^*),
\]
the latter quantity being finite. 

Since \eqref{LB} has been established in Proposition \ref{LDPLBp}, we just have to show \eqref{UB}. We note for convenience $T_*\B = \big\{(T_*\mu,T_*\nu) : \;(\mu,\nu)\in\B\big\}$  when $\B\subset \M_1(\R_+)\times\E(\R_-)$. For any closed set $\F\subset\M_1(\R_+)\times\E(\R_-)$ we have
\eq
\label{C1}
 \p_N\Big((\mu^N,\nu^N)\in\F\Big)  \leq  \p_N\Big((T_*\mu^N,T_*\nu^N)\in{\rm clo}(T_*\F)\Big),
\qe
where ${\rm clo}(T_*\F)$ stands for the closure of $T_*\F$ in $\M_1(\s_+)\times\E(\s_-)$. Then, since $\M_1(\s_+)\times\E(\s_-)$ is compact so is ${\rm clo}(T_*\F\big)$ and, by extracting a finite covering of ${\rm clo}(T_*\F)$ from an  appropriate  covering by balls, a classical argument from LDPs theory (see for example the proof of \cite[Theorem 4.1.11]{DZ})  yields from Proposition \ref{LDPUB} that
\eq
\label{C2}
\limsup_{N\rightarrow\infty}\frac{1}{N^2}\log\Big\{ Z_N \p_N\Big((T_*\mu^N,T_*\nu^N)\in{\rm clo}(T_*\F)\Big)\Big\}
\leq -\inf_{(\mu,\nu)\in{\,\rm clo}(T_*\F)} J(\mu,\nu).
\qe
If $(\mu,\nu)\in{\rm clo}(T_*\F)$ is such that $\mu(\{(0,1)\})=0$, then $(\mu,\nu)\in T_*\F$. Indeed, let $\big((T_*\eta_N,T_*\lambda_N)\big)_N$ be a sequence in $T_*\F$ with limit $(\mu,\nu)$ satisfying $\mu(\{(0,1)\})=0$. Since $T_*$ is an homeomorphism from $\M_1(\R_+)$ (resp. $\E(\R_-)$) to $\big\{\mu\in\M_1(\s_+) : \;\mu(\{(0,1)\})=0\big\}$ (resp. $\E(\s_-)$), this provides $(\eta,\lambda)\in\M_1(\R_+)\times\E(\R_-)$ such that $(\mu,\nu)=(T_*\eta,T_*\lambda)$ and moreover the convergence of  $\big((\eta_N,\lambda_N)\big)_N$  towards  $(\eta,\lambda)$. Since $\F$ is closed necessarily  $(\mu,\nu)\in T_*\F$. 

As a consequence, because $J(\mu,\nu)=+\infty$ as soon as $\mu(\{(0,1)\})>0$, we obtain from the relation  \eqref{relation}
\eq
\label{C3}
\inf_{\mu\in{\,\rm clo}(T_*\F)} J(\mu,\nu)   = \inf_{\mu\in T_*\F}J(\mu,\nu) = \inf_{\mu\in \F}\J(\mu,\nu).
\qe
Finally,  \eqref{UB} follows from \eqref{C1}--\eqref{C3}. The proof of Theorem \ref{generalLDP} is therefore complete.

\end{proof}

\subsection*{Acknowledgments}
The first author would like to thank Myl\`ene Maida for useful detailed explanations concerning her work \cite{GM}. He also would like to thank Michel Ledoux for his advice and generous encouragement. 

The authors are supported by FWO-Flanders projects G.0427.09 and by the Belgian Interuniversity
Attraction Poles P6/02 and P7/18.

The second author is also supported by FWO-Flanders projects G.0641.11 and G.0934.13, by K.U. Leuven research grant OT/08/33 and OT/12/73, and by research grant MTM2011-28952-C02-01 from the Ministry of Science and Innovation of Spain and the European Regional Development Fund (ERDF).

\end{document}